\documentclass{amsart}
\usepackage{amsmath, amssymb, amsthm}
  \usepackage{enumitem,verbatim}
\usepackage{xypic}

 \usepackage[colorlinks=true]{hyperref}
 \usepackage{breakurl}
\hypersetup{urlcolor=blue, citecolor=red}

\newtheorem{theorem}{Theorem}[section]
\newtheorem{corollary}[theorem]{Corollary}
\newtheorem{lemma}[theorem]{Lemma}
\newtheorem{proposition}[theorem]{Proposition}

\newtheorem{claim}[theorem]{Claim}
\newtheorem{vthm}{Theorem}

\theoremstyle{definition}
\newtheorem{definition}[theorem]{Definition}
\newtheorem*{remark}{Remark}
\newtheorem{refrem}[theorem]{Remark}

\renewcommand{\limsup}{\varlimsup}
\renewcommand{\emptyset}{\varnothing}
\newcommand{\restrict}[2]{{#1}{\restriction_{ #2}}}
\newcommand{\restrictThm}[2]{{#1}\! \!  \restriction_{\! #2}}

\newcommand{\R}{\mathbb {R}}

\newcommand{\Z}{\mathbb {Z}}
\newcommand{\N}{\mathbb {N}}
\newcommand{\T}{\mathbb {T}}

\newcommand{\inv}{^{-1}}
\newcommand{\orbit}{\mathcal {O}}
\newcommand{\Rec}{\mathcal{R}}
\newcommand{\Mat}{\mathrm{Mat}}

\newcommand{\Fols}{\mathcal{W}^s}

\newcommand{\stab}[1]{W^s({#1})}
\newcommand{\unst}[1]{W^u({#1})}
\newcommand{\locStab}[2][\epsilon]{W^s_{#1}(#2)}
\newcommand{\locUnst}[2][\epsilon]{W^u_{#1}(#2)}

\renewcommand{\int}{\mathrm{int}}

\newcommand{\diam}{\mathrm{diam}}

\newcommand{\NW}{\mathrm{NW}}

\newcommand{\supp}{\mathrm{supp}}
\newcommand{\td}{\tilde}
\newcommand{\wtd}{\widetilde}
\newcommand{\sol}{\mathcal S}

\title[Nonexpanding Attractors]{Nonexpanding Attractors:  Conjugacy to Algebraic Models and Classification in 3-Manifolds}


\author[A. W. Brown]{Aaron W. Brown}
\address{Department of Mathematics, Tufts University, Medford, MA 02155}
 \email{aaron.brown@tufts.edu}

\begin{document}




\begin{abstract}
We prove a result motivated  by Williams's classification of expanding attractors  and the Franks-Newhouse Theorem on co\-dimen\-sion-1  Anosov diffeomorphisms: If $\Lambda$ is a topologically mixing hyperbolic attractor with $\dim \restrict{E^u}{\Lambda} = 1$ then either $\Lambda$ is expanding or is homeomorphic to a compact abelian group (a \emph{toral solen\-oid}); in the latter case $\restrict f \Lambda$ is conjugate to a group automorphism.  
As a corollary we obtain a classification of all 2-dimensional basic sets in 3-man\-ifolds.  Furthermore we classify all topologically mixing hyperbolic attractors in 3-man\-ifolds in terms of the classically studied examples, answering a question of Bonatti in \cite{BonattiEmail}.  
\end{abstract}

\maketitle

\section{Introduction}
In the study of hyperbolic dynamics, a major theme is that strong dynamical hypotheses impose a conjugacy between an abstract  dynamical system and an algebraic, or at least highly structured, model.  For instance, results of Franks and Manning established that every Anosov diffeomorphism of an infranil-man\-ifold is conjugate to a hyperbolic infranil-automorphism  \cite[Theorem C ]{Manning:1973p11377}.  Among of the oldest conjectures in modern dynamics is the hypothesis that every Anosov diffeomorphism is conjugate to a hyperbolic infranil-automorphism.  A partial result towards this conjecture was obtained by Franks and Newhouse for co\-dimen\-sion-1 Anosov systems.  Recall an Anosov diffeomorphism is called \emph{co\-dimen\-sion-1} if $\dim (E^\sigma) = 1$ for some $\sigma \in \{s, u\}$.
\begin{vthm}[Franks-Newhouse \cite{MR0271990}, \cite{MR0277004}; see also \cite{MR1836432}]\label{thm:fn}
Let $f\colon M\to M$ be a co\-dimen\-sion-1 Anosov diffeomorphism.  Then $M$ is homeomorphic to a torus, and $f$ is conjugate to a hyperbolic toral automorphism.  
\end{vthm}

Outside the realm of global hyperbolicity, that is, when dealing with  proper hyperbolic subsets $\Lambda\subset M$,  
one often sees dynamics which is not conjugate to any algebraic system.  
However, in the case of expanding attractors, Williams showed in \cite{MR0348794} that the restricted dynamics $\restrict f \Lambda$ is conjugate to the shift map on a generalized solen\-oid.  Recall that by an \emph{expanding attractor} we mean a hyperbolic attractor $\Lambda$ such that $\dim(\Lambda) = \dim (\restrict {E^u} \Lambda)$. 
Also  by a \emph{generalized solen\-oid} (or $n$-solen\-oid) we mean a topological space $N$ (which Williams takes to be a  \emph
{branched $n$-man\-ifold}), and a surjective map $g\colon N \to N$, and define the generalized solen\-oid to be the inverse limit 
\[\varprojlim(N,g) :=  \varprojlim\{N \xleftarrow{g} N \xleftarrow{g} N \xleftarrow{g} \dots \}\]
with the natural shift map $\sigma$.  (See Section \ref{sec:AS} for the construction of the inverse limit in a more specific setting adapted to our problem.)   
\begin{vthm}[{\cite[Theorem A]{MR0348794}}]\label{thm:Williams}

Assume $\Lambda$ is an $n$-dimensional expanding attractor for $f\in \mathrm{Diff}(M)$. 
Then $\restrict f \Lambda$ is conjugate to the shift map of an $n$-solen\-oid.\end{vthm}

Note that Theorem \ref{thm:Williams}, as originally stated in \cite{MR0348794}, required the additional hypothesis that the foliation $\{\locStab x \mid x\in \Lambda\}$ was $C^1$ on some neighborhood of $\Lambda$.  This was latter seen to be unnecessary (see for example \cite{MR1179171}).  
While not algebraic, the conjugacy in Theorem \ref{thm:Williams} provides a significant insight into the topology of $\Lambda$ and the dynamics of $\restrict f \Lambda$.    

In this article we present a result inspired in part by the Franks-Newhouse Theorem on co\-dimen\-sion-1 Anosov diffeomorphisms, and somewhat dual to the conjugacy between the dynamics of 1-dimensional expanding attractors and shift maps on generalized solenoids established in \cite{MR0217808} and \cite{MR0266227}.
In particular, we study \emph{nonexpanding} hyperbolic attractors $\Lambda$ for an embedding $f$, under the assumption that  $\restrict{\dim E^u}{\Lambda} =1$, and show that the dynamics $\restrict f \Lambda$ is  conjugate to an automorphism of a compact abelian group.  We take our dynamics to be generated by $C^r$ embeddings for $r\ge 1$.  
\begin{theorem}\label{thm:main}
Let $\Lambda\subset U\subset M$ be a compact topologically mixing hyperbolic attractor for a $C^r$ embedding  $f\colon U \to M$ such that $\restrict{\dim E^u} \Lambda =1$.  Then either $\Lambda$ is expanding, or is an embedded  toral solenoid (see Section \ref{sec:AS}).  In the latter case,  $\restrict {f} \Lambda$ is conjugate to a leaf-wise hyperbolic solen\-oidal automorphism.  In particular, if $\Lambda$ is locally connected then $\Lambda$ is homeomorphic to a torus and $\restrict f \Lambda$ is conjugate to a hyperbolic toral automorphism.  
\end{theorem}

Using the primary result in \cite{MR2415057} we conclude that the only 2-dimensional toral solenoids that may be embedded in a $3$-man\-ifold are homeomorphic to $\T^2$.  
In particular, we obtain the following.

\begin{corollary}\label{cor:dim3}
Let $M$ be a 3-man\-ifold, and let $\Lambda\subset M$ be a basic set with $\dim (\Lambda )=2$.  Then either $\Lambda$ is a co\-dimen\-sion-1 expanding attractor (or contracting repeller), or $\Lambda$ decomposes as a disjoint union $$\Lambda = \Omega_1 \cup \Omega_2 \cup \dots \cup \Omega_k$$ where each $\Omega_j$ is 
homeomorphic to $\T^2$ and $\restrictThm {f^k} {\Omega_j}$ is conjugate to a hyperbolic automorphism of $\T^2$.  
\end{corollary}
We note that the above corollary is a significantly stronger version of the main result in \cite{MR2249032}.  Indeed, in  \cite{MR2249032} the result  corresponding  to the second case in Corollary \ref{cor:dim3} requires the additional hypothesis that $\Lambda$ is embedded as a subset of a closed surface in $M$.  Our result, on the other hand,  rules out the possibility that $\dim (\restrict {E^u}{\Lambda} ) = 1$ and $\stab x \cap \Lambda$ is a connected 1-dimensional set that is not a manifold, for example, a  Sierpinski carpet.  

It should also be noted that in the conclusion of Corollary \ref{cor:dim3}, the $\T^2$ need not be smoothly embedded.  Indeed in \cite{MR766105} a hyperbolic attractor is constructed as a nowhere differentiable  torus embedded in a 3-man\-ifold.  

The motivation for this work was initially to  answer a question by  Bonatti \cite{BonattiEmail} which can be paraphrased as follows: Do there exist examples of hyperbolic attractors in 3-man\-ifolds besides the classical examples?  We answer this question in the negative. 
\begin{theorem}\label{thm:HA3D}
Let $M$ be a $3$-man\-ifold, and let $\Lambda\subset U\subset M$ be a topologically mixing, hyperbolic attractor for a $C^r$ embedding $f\colon U \to M$.  If 
\begin{itemize} \item[] \begin{description}
\item[$\dim \Lambda =0$]  then $\Lambda$ is an attracting fixed point  for $f$;
\item[$\dim \Lambda = 1$] then $\dim \restrict {E^u}\Lambda = 1 $ and $\Lambda$ is conjugate to the shift map on a generalized $1$-solen\-oid as classified by Williams (\cite {MR0266227});
\item[$\dim \Lambda = 2$] then we have $1\le \dim \restrict {E^u}\Lambda \le 2 $ and if
\begin{description}
\item[$\dim \restrict {E^u}\Lambda = 1$] then
$\Lambda$ is homeomorphic to $\T^2$ and $\restrict f \Lambda$ is conjugate to a hyperbolic toral automorphism;
\item[$\dim \restrict {E^u}\Lambda = 2 $] then $\Lambda$ is a co\-dimen\-sion-1 expanding attractor studied by Plykin (\cite{MR580625}, \cite{MR586207});
\end{description}
\item[$\dim \Lambda = 3$] then $\Lambda= M\cong \T^3$ and $f$ is conjugate to a hyperbolic toral automorphism. 
\end{description}
\end{itemize}
\end{theorem}

We remark that in the case of 1-dimensional topologically mixing attractors (which are necessarily expanding), the proof of Theorem \ref{thm:main} provides a mechanism to determine if the attractor is algebraic, that is,  if $\restrict f \Lambda$  is conjugate to a solen\-oidal automorphism.  In particular the presence or absence of a \emph{global product structure} as described in Section \ref{section:GPS} determines whether or not a 1-dimensional attractor is algebraic.  See Proposition \ref{prop:MainAlg}.

\section{Hyperbolic dynamics}

We begin with background material in hyperbolic dynamics and attractors.
Let $M$ be a smooth manifold endowed with a Riemannian metric. Given $U\subset M$ and a $C^r$ embedding  $f\colon U\to M$, $r\ge 1$, we say a subset $\Lambda\subset U$ is  \emph{invariant} if $f(\Lambda) = \Lambda$.    A compact invariant set $\Lambda$ is said to be \emph{hyperbolic} if  there exist a Riemannian metric on $M$ (called the \emph{adapted} metric), a constant $\kappa<1$,  and a continuous $Df$-invariant splitting of the tangent bundle $T_x M = E^s(x)\oplus E^u(x)$ over $\Lambda$   
so that for every  $x\in \Lambda$ and $n \in \N$
\begin{align*}
 \|Df^n_x v\|\le \kappa^n \|v\|,  \quad &\mathrm{ for } \ v\in E^s(x)\\
\|Df^{-n}_x v\|\le \kappa^{n}\|v\|,  \quad &\mathrm{ for }\   v\in E^u(x).
\end{align*}

We set \[V^\pm = \bigcap _{n\in \N} f^{\pm n} (U).\]  
When $\Lambda$ is hyperbolic, there exists an $\epsilon>0$ such that the sets 
\begin{align*}
W^s_\epsilon (x) &:= \{y\in V^- \mid d(f^n(x),f^n(y))< \epsilon, \  \mathrm{for\ all} \  n\ge 0\} \\ W^u_\epsilon (x) &:= \{y\in V^+\mid d(f^{-n}(x),f^{-n}(y))< \epsilon, \ \mathrm{for\ all} \  n\ge 0\}
\end{align*}
are $C^r$ embedded open disks, called the \emph{local stable} and \emph{unstable} manifolds. 
Furthermore, if $d$ is the distance on $M$ induced by  the adapted metric, there are $\lambda<1<\mu$ so that for $x\in \Lambda, y\in \locStab[\epsilon]x, z\in \locUnst[\epsilon]x$ and $n\ge 0$ we have
\begin{align}
d(f^n(x), f^n(y))& \le \lambda^n d(x,y)\label{eqn:asymS}\\
d(f^{-n}(x), f^{-n}(z))& \le \mu^{-n} d(x,z).\label{eqn:asymU}
\end{align}
Note that \eqref{eqn:asymS} and \eqref{eqn:asymU}  imply  $ f(W_{\epsilon}^s(f\inv(x)))\subset W_{\epsilon}^s(x)$ and $W_{\epsilon}^u(x) \subset f(W_{\epsilon}^u(f\inv(x)))$. 

For $x\in \Lambda$ we also have the sets 
\[\stab x := \{y\in V^-\mid d(f^n(x), f^n(y) )\to 0 \mathrm{ \ as \ } n\to \infty\} \]
and 
\[\unst x := \{y\in V^+\mid d(f^{-n}(x), f^{-n}(y) )\to 0 \mathrm{ \ as \ } n\to \infty\} \]
called the \emph{global} stable and unstable manifolds.  
Both $\unst x$ and $\stab x$ are $C^r$ injectively immersed submanifolds.  
Note that in the case that $f$ is invertible (that is, when $f(U) = U$),  we have $\unst x \cong \R^{\dim E^u(x)}$ and $\stab x \cong \R^{\dim E^s(x)}$. 

An invariant set $\Lambda$ is said to be \emph{topologically transitive} under $f$ if it contains a dense orbit.  Alternatively, a compact invariant subset $\Lambda\subset M$ is topologically transitive if for all pairs of nonempty open sets $U, V\subset \Lambda$, there is some $n$ such that $f^n(U)\cap V \neq \emptyset$. An invariant set $\Lambda$ is called \emph{topologically  mixing} if for all pairs of nonempty  open sets $U, V \subset \Lambda$, there is some $N$ such that $f^n(U)\cap V \neq \emptyset$ for all $n\ge N$.

A hyperbolic set $\Lambda$ is a \emph{hyperbolic  attractor} if there is some open neighborhood $ \Lambda\subset V$ such that $\bigcap_{n\in \N}f^n(V) = \Lambda$.  Alternatively, if $\Lambda$ is a hyperbolic set, then it is an attractor if and only if $\unst x \subset \Lambda$ for all $x\in \Lambda$. When $\Lambda$ is a hyperbolic attractor,  the set $\bigcup_{y\in \Lambda}\stab y$ is called the \emph{basin} of $\Lambda$. Note that if $\Lambda$ is a topologically  mixing hyperbolic attractor, then for each $x\in \Lambda$, $\unst x$ is dense in $\Lambda$.  

We recall from the introduction that a hyperbolic attractor $\Lambda$ is called \emph{expanding} if the topological dimension of $\Lambda$ equals the dimension of the unstable manifolds.  (For an introduction to topological dimension, see \cite{MR0006493}.)  Alternatively, $\Lambda$ is expanding if for every $x\in \Lambda$ the set $\locStab x \cap \Lambda$ is totally disconnected.

\subsection{Local product structure and Markov partitions}

Recall that given a compact hyperbolic set, we may find  $0<\delta<\eta$   so that $d(x,y)<\delta $ implies the intersection $ \locUnst[\eta] x \cap \locStab [\eta] y $ is a singleton.  We say that a hyperbolic  set $\Lambda$ has \emph{local product structure} if for $\eta, \delta$ above,  $d(x,y)<\delta $ implies  $ \locUnst[\eta] x \cap \locStab [\eta] y \subset \Lambda$.  
A compact hyperbolic set $\Lambda$ is called \emph{locally maximal} if there exists an open set $\Lambda\subset V $ such that $\Lambda = \bigcap_{n\in \Z} f^n(V).$  For compact  hyperbolic sets, local maximality is equivalent to the existence of a 
{local product structure} \cite{MR1326374}; in particular, hyperbolic attractors have local product structure.  

\begin{definition}\label{def:LPC} Given  a set $\Lambda$ with local product structure and $\delta$ and $ \eta$ as above, we say a closed set $R\subset \Lambda$ is a \emph{rectangle} or a \emph{local product chart} if 
\begin{enumerate}
\item $\sup\{d(x,y)\mid x,y\in R\} <\delta$;
\item $R$ is proper, that is, $R $ is equal to the closure of its interior (in $\Lambda$);
\item $x,y\in R $ implies $\locUnst[\eta] x \cap \locStab[\eta] y \subset R$.  
\end{enumerate}
If $R$ is a rectangle, we write $W^\sigma_R(x) :=W^\sigma_\eta(x) \cap R$.  

If $\Lambda$ is an attractor, we say an ambiently open set $V\subset M$ is a \emph{local $s$-product neighborhood} if the closure of $V\cap \Lambda$ is a rectangle and for each $x\in V\cap \Lambda$ $$V\subset \bigcup_{y\in \locUnst[\delta]x} \locStab[\delta] y.$$  For $x\in V\cap \Lambda$ we notate $$\locStab[V] x:= \locStab[\delta] x \cap V.$$

\end{definition} 

\begin{definition}\label{def:markov}  Given a hyperbolic set $\Lambda$ with local product structure we say a collection of rectangles $\mathcal R = \{R_j\}$  is a \emph{Markov partition} if 
\begin{enumerate}
	\item $\Lambda = \bigcup_j R_j$;
	\item for $i\neq j$, $R_j \cap R_i\subset \partial R_j$ where $\partial $ denotes the topological boundary;
	\item $x\in R_i \in \mathcal R$ and $f(x) \in \int(R_j)$ implies  \[f(\locStab[R_i] x) \subset R_j\] and $f\inv(x) \in \int(R_j)$ implies  \[f\inv(\locUnst[R_i] x) \subset R_j.\]
\end{enumerate}
\end{definition}
We note that every locally maximal hyperbolic set admits a Markov partition; in particular, hyperbolic attractors admit Markov partitions.  Also note that if $R$ is a rectangle and $\mathcal R$ is a Markov partition, then $f(R)$ is a rectangle and $f(\mathcal R):= \{f(R_j)\}$ is a Markov partition.  In particular, we have the following.
\begin{claim}\label{clm:compactInside}
If $\Lambda$ is locally maximal, then given any  set $K\subset W^\sigma(x) \cap \Lambda$, compact in the internal topology of $W^\sigma(x)$, there is a rectangle containing $K$.  
\end{claim}

\subsection{Disintegration of the measure of maximal entropy}
For a hyperbolic set with local product structure we define a \emph{canonical isomorphism} between subsets of the stable and unstable  manifolds.  

\begin{definition}[Canonical Isomorphism]\label{def:CI}  Let $\Lambda$ be a locally maximal hyperbolic set, $R$  a rectangle, and $x\in R$.  Let $x'\in \locUnst  [R] x$, and let $D\subset \locStab  [R] x$, $D'\subset \locStab[R] {x'}$.  Then $D$ and $D'$ are said to be \emph{canonically isomorphic} 
if $y\in D\cap \Lambda$  implies $D'\cap \locUnst[R] y \neq \emptyset$ and $y'\in D'\cap \Lambda$ implies $D \cap \locUnst [R] {y'} \neq \emptyset$. 

Similarly, we may define a canonical isomorphism between subsets of local unstable manifolds.  
\end{definition}

Recall that a point $x\in M $ is said to be \emph{nonwandering} if for every neighborhood $ U$of $x$, there is an $n$ so that $f^n(U) \cap U \neq \emptyset$.  Let $\NW(f)$ denote the nonwandering points of $f$.  
Recall that given an Axiom-A diffeomorphism, (respectively a locally maximal hyperbolic set 
$\Lambda = \bigcap _{n\in \Z} f^n(V)$) we have a partition, called the \emph{spectral decomposition}, of the nonwandering points $\NW(f)= \Omega_1 \cup \dots \cup \Omega_k$ (respectively $\NW(\restrict f \Lambda)=  \Omega_1 \cup \dots \cup \Omega_k$) where each $\Omega_j$ is a transitive hyperbolic set for $f$ (see  \cite{MR1326374}, \cite{MR0228014}).  Given a spectral decomposition, we call  the partition elements $\Omega_j$ above \emph{basic sets}.  That is, a compact hyperbolic set $\Omega\subset \NW(f)$ is a \emph{basic set}, if $\Omega$ is open in $\NW(f)$ and $f$ is topologically transitive on $\Omega$.  Clearly  topologically mixing hyperbolic attractors are  basic sets.  

Given a basic set, there is a canonical disintegration of the measure of maximal entropy as a product of measures supported on the stable and unstable manifolds.  The following is adapted from  \cite{MR0415679}.  
\begin{theorem}[Ruelle, Sullivan  \cite{MR0415679}]\label{thm:RSS}
Let $\Omega$ be a basic set for $f$.  
Let $h $ be the topological entropy of $\restrict f \Omega$. Then there is an $\epsilon>0$ so that for each $x\in \Omega$ there is a measure $\mu_x^u$ on $\locUnst x$ and a measure $\mu_x^s$ on $\locStab x$ such that: 

\begin{enumerate}[ label=\emph{\alph*)}, ref=(\alph*) ]

\item $\supp  (\mu_x^u)= \locUnst x\cap \Omega$ and $\supp ( \mu_x^s)= \locStab x\cap \Omega$,
\item \label{RSS:CI} $\mu^u_x $ and $\mu^s_x$ are invariant under canonical isomorphism (see Definition \ref{def:CI}); that is, if $x'\in \locStab [\eta] x$ and $D\subset \locUnst [\eta]  x, D'\subset \locUnst  [\eta] {x'}$ are canonically isomorphic then $\mu^u_x(D )= \mu^u_{x'} (D')$,  and if  $x'\in \locUnst [\eta] x$ and $D\subset \locStab  [\eta] x, D'\subset \locStab  [\eta]{ x'}$ are canonically isomorphic then $\mu^s_x (D )= \mu^s_{x'} (D')$,
\item $f_* \mu ^u_x = e^{-h} \mu ^u_{f(x)} $  and  $f\inv _* \mu ^s_x = e^{-h}  \mu ^s_{f\inv(x)} $,
\item the product measure $\mu^u_x \times \mu^s_x$ is locally equal to Bowen's measure of maximal entropy.  
\end{enumerate}
\end{theorem}
By \ref{thm:RSS}\ref{RSS:CI} we drop the subscript and simply write $\mu^\sigma$.  By additivity, we may extend the definition of $\mu^\sigma$ to any set $K\subset W^\sigma (x)$ for $\sigma\in \{s,u \}$.  The following properties of $\mu^\sigma$ are corollaries to the proof of Theorem \ref{thm:RSS} in \cite{MR0415679}.  
\begin{corollary}\label{cor:toRSS}
Let $\Omega$ be a basic set with an infinite number of points.   Then for $\sigma\in \{s,u\}$
\begin{itemize}
\item[---] $\mu^\sigma$ is non-atomic and positive on non-empty open sets in $W^\sigma( x)\cap \Lambda$;
\item[---] $\mu^\sigma(K)$ is finite for  sets $K\subset W^\sigma( x)$ compact in the internal topology of $W^\sigma(x)$.
\end{itemize}
\end{corollary}
Furthermore, in the case $\restrict{\dim E^u} {\Lambda} =1$, we have the following.
\begin{corollary}\label{cor:toRSS2}
Let $\Lambda$ be a hyperbolic attractor such that $\restrict{\dim E^u} {\Lambda} =1$.  Then for any  connected set  $K\subset \unst x$ (that is, an interval) we have  $\mu^u(K)<\infty$
if and only if its closure $\overline {K}$ in $\unst x$ is compact in the internal topology of $\unst x$.
\end{corollary}
\begin{proof}
If $\overline K$ is not compact in $\unst x$, then $K$ passes through some  rectangle  a countable number of times which implies that $\mu^u(K)=\infty$. 
\end{proof}

\section{Limits of directed and inverse systems}
We review basic constructions and properties of the direct and inverse limit objects in algebra and topology. 
\subsection{Direct limits}\label{sec:dir}
Given a topological space $X$ and an injective continuous map $f\colon X\to X$ we construct the \emph{direct  limit} 
\[\varinjlim(X, f):=\varinjlim\{X \xrightarrow{f}X \xrightarrow{f}X \xrightarrow{f} \dots \}\]
as follows.  Endow $\N$ with the discrete topology and introduce the equivalence relation on $X \times \N$ generated by the relation $(x, k) \sim (f(x), k+1)$.  Then we define 
\[\varinjlim(X,f):=(X \times \N)/{\sim}.\]  
The map $f\colon X\to X$ naturally induces a homeomorphism $\tau_f\colon \varinjlim(X,f) \to\varinjlim(X,f)$ by \[\tau _f\colon [(x, m)]\mapsto [(f(x), m)].\]
Note, that for $m\ge 1$ we have $\tau_f ([(x, m)] ) = [(x, m-1)]$, whence it is natural to refer to $\tau_f$ as the \emph{left shift} on $\varinjlim(X, f)$.  

We present an alternate, more explicit, construction of the set $\varinjlim(X, f)$. 
For every $j\in \N$ define a homeomorphism \[h_j\colon X_j \to X\] and consider the inclusion $i_j\colon X_j \hookrightarrow X_{j+1}$ given by $ i_j = h_{j+1} \inv \circ f\circ h_j$. 
We then have $X_0\subset X_1\subset X_2\subset \dots$ whence we define \begin{align}\label{eq:DL} \varinjlim\{X  \xrightarrow{f}X \xrightarrow{f}\dots \} = \bigcup_{n\in \N} X_n.\end{align} 
When the map $f\colon X\to X$ is open, the inclusions  $i_j$ induce a nested inclusion of topologies, the union of which correctly reconstructs the topology of the direct limit.  
 Given $\xi\in\varinjlim(X,f)$, we have that $\xi \in X_j$ for some $j$ whence we may define \[\tau_f (\xi) = h_j\inv \circ f \circ h_j(\xi).\]  One verifies this definition of $\tau_f$ coincides with that above.  

By the second construction, we see that if $X$ is a $C^r$ manifold and $f$ a $C^r$ embedding, then $\varinjlim(X,f)$ can be endowed with a $C^r$ differential structure under which $\tau_f$ is a $C^r$ diffeomorphism.  

Given a group $G$ and a homomorphism $h\colon G \to G$ we define 
\[\varinjlim(G, h) :=\varinjlim\{G \xrightarrow{h}G \xrightarrow{h}G \xrightarrow{h} \dots \}\]
as follows.  Let $i_j\colon G_j\to G$ be a group isomorphism and define $h_j \colon G_j \to G_{j+1}$ by $h_j\colon g\mapsto i_{j+1} \inv (h(i_j(g)))$.  
Let $N$ be the normal subgroup of $\bigoplus_{k \in \N} G_k$ generated by the elements $\{ g_j\inv h_j(g_j)\}$ for $g_j \in G_j$.   
Then 
\[\varinjlim(G, h)  = \Big(\bigoplus_{k \in \N} G_k\Big) / N \]
with canonical left shift automorphism $\tau_h$ given by $\tau_h\colon [g_j] \mapsto  [i_j \inv \circ h\circ i_j(g_j)]$.  We notate $[(g,m)] := g_m+ N$ for $g_m\in G_m$.  

The following proposition is straightforward from the Van Kampen theorem.  
\begin{proposition}
Let $X$ be a connected manifold, $f\colon X\to X$ an embedding, and $G= \pi_1(X)$.  Then \begin{itemize}
\item[---] $\pi_1(\varinjlim(X,f)) = \varinjlim(G, f_*)$
\item[---] $(\tau_f)_*$ is the map given by \((\tau_f)_*([(g,m)]) = [(f_*(g),m) ].\)
\end{itemize}
\end{proposition} 
The construction above allows us to embed every hyperbolic attractor as an attractor for an ambient diffeomorphism.  
\begin{claim}\label{rem:Intertible} 
Let $\Lambda\subset U\subset M$ be a hyperbolic attractor for a $C^r$ embedding $f\colon U\to M$.  Then there is a  $C^r$ diffeomorphism $f'\colon M'\to M'$, a hyperbolic attractor $\Lambda'\subset M'$ for $f'$, neighborhoods $N$ and $N'$ of $\Lambda$ and $\Lambda'$ respectively, and a $C^r$ diffeomorphism $h\colon N\to N'$ so that $h(\Lambda) = \Lambda'$ and $h\circ \restrict f N = \restrict {f'}{N'} \circ h$.
\end{claim}
\begin{proof}
Let $N= \bigcup_{y\in \Lambda} \locStab y$.  Then  $f\colon N\to N$ is a $C^r$ embedding.  Take $X = N$ and $M':= \varinjlim (X, f)$.  Then we have a canonical inclusion $\Lambda \subset X_0\cong N$, where $X_0$ is as in \eqref{eq:DL}.  But then $\Lambda \subset X_0 \subset M' $ is a hyperbolic attractor for the $C^r$ diffeomorphism $\tau_f\colon M'\to M'$.  
\end{proof}

Note that in constructing the direct limit, we assumed the map  $f\colon X\to X$  was injective to avoid pathological topological properties in the limiting object.  

\subsection{Inverse limits}
Let $f\colon X\to X$ be a continuous map (which we typically take to be surjective).  We then define the \emph{inverse limit}   \[\varprojlim(X,f):= \varprojlim\{X \xleftarrow{f}X \xleftarrow{f}X \xleftarrow{f} \dots \}\] to be the subset of $X^\N:=\prod_{i\in \N} X$
satisfying \[(x_0, x_1, x_2, \dots) \in \varprojlim(X,f) \mathrm{\  if\ } x_j = f(x_{j+1})\] for all $j\in \N$.  We then have an induced homeomorphism $\sigma_f\colon \varprojlim(X,f)\to \varprojlim(X,f)$ given by \[\sigma_f\colon(x_0, x_1, x_2, \dots) \mapsto(f(x_0), f(x_1), f(x_2), \dots) =(f(x_0), x_0, x_1, x_2, \dots)\] hence it is natural to call $\sigma_f$ a \emph{right shift} map.  
We will call the topological object $\varprojlim(X,f)$ a \emph{generalized solen\-oid}.  

Note that even in the case that $X$ is a manifold and $f$ is a smooth map, we do not expect $\varprojlim(X,f)$ to have a manifold structure.  Indeed in the case that $f$ is a $C^\infty$ covering with degree  greater than $1$, the limit $\varprojlim(X,f)$ will locally be the product of a Cantor set and a manifold.  

If $G$ is a group, and $h$ a 
homomorphism we define 
\[\varprojlim(G,h):= \varprojlim\{G\xleftarrow{h}G \xleftarrow{h}G \xleftarrow{h} \dots \}\]
to be the subgroup of $\prod_{n\in \N} G$ satisfying  $(g_0, g_1, g_2, \dots ) \in \varprojlim(G,h) $ if $g_j = h(g_{j+1})$, with the induced right shift automorphism \[\sigma_h\colon (g_0, g_1, g_2, \dots )\mapsto (h(g_0), g_0, g_1, g_2, \dots ).\]


\section{Toral solen\-oids}\label{sec:AS}
We give a brief introduction to \emph{toral solen\-oids}, the compact abelian groups obtained as the algebraic models in the conclusion of Theorem \ref{thm:main}.  For more detailed exposition, see, for example, \cite{MR1407604}. 
For an explicit construction of toral solen\-oids embedded as hyperbolic attractors for differentiable dynamics, see \cite{MR1254137}.

Let $A\in \Mat(k, \Z)$ have non-zero determinant.  Then considering the standard torus $\T^k :=\R^k/\Z^k$ as a compact abelian group, the map $A\colon \R^k\to \R^k$ induces an  endomorphism $A\colon \T^k\to \T^k$.  We define a \emph{toral solen\-oid} $\sol_A$ to be the topological group obtained via the inverse limit 
\[\sol_A :=\varprojlim(\T^k,{A}) = \varprojlim\{\T^k \xleftarrow{A} \T^k\xleftarrow{A} \dots \}. \]
Note the above inverse limit is taken both as a limit of topological and algebraic objects, and that $\sol_A$ inherits the right shift automorphism $\sigma_A\colon \sol_A \to \sol_A$ \[\sigma_A\colon(x_0, x_1, x_2, \dots) \mapsto (Ax_0, x_0, x_1, \dots).\]

$\sol_A$ will fail to be a manifold in the case when $|\det (A)|>1$, in which case we will call $\sol_A$ \emph{proper}.  Let $\mathcal C_\xi$ denote the path component of $\xi$ in $\sol_A$.  Then, even in the case $|\det (A)|>1$, we can endow the  path components $\{\mathcal {C}_\xi\} $ with the smooth Euclidean structure pulled back from the projection to the zeroth coordinate $\sol_A\to\T^k$.  With respect to this Euclidean structure the map $\sigma_A\colon \mathcal C_\xi \to \mathcal C_{\sigma_A(\xi)}$ is smooth.  
Furthermore, in the case when $A$ has no eigenvalues of modulus 1, the map $\sigma_A\colon \mathcal C_\xi \to \mathcal C_{\sigma_A(\xi)}$ is  hyperbolic with respect to the pull-back metric, whence we say $\sigma_A$ is \emph{leaf-wise hyperbolic}.

\subsection{$\R^k$ and $\Z^k$ actions}
We now define an $\R^k$ action  and an induced $\Z^k$ action on $\sol_A$.  (Compare to the immersions  of 
$\R^k$ and $\Z^k$ into $\sol_A$ constructed in \cite{MR1407604}).

\begin{definition}\label{def:action}

We define the $\R^k$ action $\theta\colon \R^k\times \sol_A\to \sol_A$ by the rule
\begin{align}\label{eq:theta}
\theta_v \colon  ([y_0], [y_1], [y_2], \dots) \mapsto ([y_0+v], [ y_1+ A^{-1} v], [y_2+ A^{-2} v],  \dots)
\end{align}
for $v\in \R^k$ and $ ([y_0], [y_1], [y_2], \dots) \in \sol_A$ where $[y]$ denotes the class of $y\in \R^k$ in the quotient $\T^k = \R^k/\Z^k$.

We then define the  $\Z^k$ action $\vartheta\colon \Z^k\times \sol_A\to \sol_A$ to be the restriction of $\theta$ to the subgroup $\Z^k\subset\R^k$.  
\end{definition}
Let $p_0\colon \sol_A \to \T^k$ denote the projection in the zeroth coordinate.  
\begin{claim}\label{clm:Theta}
The action  $\theta$ has the following properties.
\begin{enumerate}[ label=\emph{\alph*)},  ref=\ref{clm:Theta}(\alph*)]
\item \label{Thetadense} For each $\xi\in \sol_A$, the $\theta$-orbit of $\xi$ 
is dense.
\item \label{ThetaEqui} The homomorphism $\sigma_A$ is $\theta$-equivariant; that is, \[\sigma_A (\theta_v(\xi)) = \theta_{Av}( \sigma_A (\xi))\] and   
\[\sigma_A \inv (\theta_v(\xi)) = \theta_{A\inv v}( \sigma_A\inv  (\xi)).\]  
\item  \label{ThetaProj} For all $\xi \in \sol_A$ and $v\in \R^k$ we have  \[  p_0(\theta_v( \xi))= p_0(\xi) + v .\]  
Here $[x] +v := [x+v]$ is the standard $\R^k$ action on $\T^k$.  
\item \label{commuteGroup} $\theta$ commutes with the group operation; that is, 
\[\theta_v(\xi+ \eta) = \theta_v(\xi) + \eta = \xi+ \theta_v(\eta)\]
\end{enumerate}
\end{claim}
\begin{proof}
\ref{Thetadense} is essentially  \cite[Proposition 2.4]{MR1407604}.  \ref{ThetaEqui}, \ref{ThetaProj}, \ref{commuteGroup} follow from \eqref{eq:theta}.  
\end{proof}

Define  $\Sigma$ to be the 0-dimensional compact group $\Sigma :=p_0 \inv\big( [0]\big)$.  Note $\Sigma$ is either the trivial group $\{1\}$ in the case that $\det (A)=\pm 1$,  or homeomorphic to a Cantor set in the case $|\det( A) |>1$.  By \cite[Corollary 2.3]{MR1407604}  the map $p_0\colon \sol_A \to \T^k$ defines a principle $\Sigma$-bundle.  

Let $p\colon \R^k\to \T^k$ denote the canonical projection.  Given an $m\in \Z^k$, we may find some curve $\gamma\colon[0,1]\to \R^k$ with $\gamma(0) = 0$ and $\gamma(1) = m$.  Then $p(\gamma)$ corresponds to a closed curve at $[0]$, hence determines an element of $\pi_1(\T^k, [0])$.  For $\eta\in \Sigma$ we have that $\gamma'(t) := \theta_{\gamma(t)} (\eta)$ is the unique lift of $p(\gamma(t))$ to $\sol_A$ starting at $\eta$.  Then we have that $\gamma'(1) = \vartheta_m(\eta).$  This motivates the construction in the next section.

\subsection{A covering space for $(\sol_A, \sigma_A)$} \label{sec:Cover}

Define the  topological group $\overline \sol $ to be the product $ \Sigma \times \R^k$. 
The group action $\vartheta$ of $ \Z^k$ on $\sol_A$ induces a group action  $\vartheta$ of $\Z^k$ on $ \Sigma$. 
We define an embedding $\alpha $ of  $\Z^k$ as a subgroup of $\overline \sol$ by 
\[\alpha(n) :=  \left(\vartheta_{-n}(e),  n \right)\]
where $e$ is the identity element of $\Sigma$.  Then $\alpha$ naturally defines a $\Z^k$ action on $\overline \sol$ by 
\[n\cdot (\xi, v) := (\xi, v) + \alpha(n) = \left(\vartheta_{-n}(\xi), v+ n \right)\]
for $n\in \Z^k, \xi\in \Sigma$, and $v\in \R^k$.

We also define  maps  $\overline\sigma\colon \overline \sol\to \overline \sol$ given by \[\overline \sigma\colon (\xi,v) \mapsto (\sigma_A (\xi), A (v))\] and  $\overline q \colon \overline \sol \to \sol_A$ given by 
\[ \overline q\colon (\xi, v) \mapsto \theta_{v}(\xi). \]
We check that  $\overline\sigma$ is an injective endomorphism.  
Furthermore,  $\overline q$ is seen to be a homomorphism by Claim \ref{commuteGroup}.

We have the following properties of the above construction.

\begin{claim} \label{claim:OVERSOL}$\overline \sol, \overline \sigma, \overline q$, and $\alpha$ satisfy  
\begin{enumerate}[ label=\emph{\alph*)}, ref=\ref{claim:OVERSOL}(\alph*)]
\item\label{PDD} $N:=\alpha(\Z^k)$  is a discrete subgroup isomorphic to $\Z^k$; 
\item\label{KER} $\ker(\overline q) = N$, whence we have the canonical identification  $\overline \sol/N\cong \sol_A$ as topological groups;
\item\label{1LIFT} $\overline q \circ \overline \sigma= \sigma_A\circ \overline q$;
\item\label{STAR} $\overline \sigma (  x+ \alpha (n) ) =  \overline \sigma(x) + \alpha(A(n))$.  
\end{enumerate}
\end{claim}

\begin{proof}
\ref{PDD} is clear.
For \ref{KER}, let $\overline q (\xi, v) = e$ where $e$ is the identity in $\Sigma\subset\sol_A$.  Then we have 
\begin{align*}\label{eq:Jujitsu}\theta_{v}(\xi) = e.\end{align*}
In particular, $v \in \Z^k$.  Furthermore \[ \alpha(v)  = (\vartheta_{-v}(e), v)  = (\theta_{v}\inv (e), v)= (\xi, v)\] hence $\overline q (\xi, v) = e $ implies $ (\xi, v)\in N$.  
Similarly for any $n\in \Z^k$ we have $\overline q(\alpha(n) )  = \vartheta_n(\vartheta_{-n}(e))=e$ hence \ref{KER} holds. 

We have
\[\overline q \circ \overline \sigma (\xi, v) = \theta_{Av}(\sigma_A \xi) = \sigma_A(\theta_{v}( \xi))= \sigma_A \circ \overline q (\xi, v)\]
whence  \ref{1LIFT} follows.
Finally we have 
\begin{align*}\overline \sigma(\alpha(n))=  (\sigma_A(\vartheta_{-n} (e)), A(n)) =  (\vartheta_{-A(n)}( \sigma_A(e)), A(n)) = \alpha({A(n)})\end{align*}
from which  \ref{STAR} follows.
\end{proof}

Now $\overline \sigma\colon \overline \sol \to \overline \sol$ is an injective homomorphism, but will fail to be surjective whenever $|\det (A)|>1.$  We define the topological group $\wtd \sol$ as the direct  limit
\[\wtd \sol := \varinjlim \left\{ \overline \sol \xrightarrow{\overline \sigma} \overline \sol \xrightarrow{\overline \sigma} \overline \sol \xrightarrow{\overline \sigma}\dots \right\}\]
and  define $\td \sigma\colon \wtd \sol \to \wtd \sol$ to be the left shift automorphism ($\tau_{\overline \sigma}$ in the notation of Section \ref{sec:dir}) induced by $\overline \sigma \colon \overline \sol \to \overline \sol$ ; that is, 
\[\td \sigma ([(s, l)]) = [(\overline \sigma(s), l)]= [(s, l-1)]\]
where the second  equality holds for $l\ge 1$.  
Furthermore, we define the torsion-free abelian group
\[\Z^k[A\inv]:=\varinjlim \left\{ \Z^k \xrightarrow{A} \Z^k \xrightarrow{A} \Z^k \xrightarrow{A}\dots \right\}\]
and the (left shift) group homomorphism $\tau_A\colon \Z^k[A\inv]\to \Z^k[A\inv]$.  
Note that by Claim \ref{STAR} the diagram 
\[\xymatrix{\Z^k \ar[d]^\alpha \ar[r]^{A}&\Z^k \ar[d]^\alpha \ar[r]^{A}&\Z^k \ar[d]^\alpha \ar[r]^{A}&\Z^k \ar[d]^\alpha \ar[r]^{A}&\dots  \\
\overline \sol\ar[r]^{\overline \sigma}&\overline \sol\ar[r]^{\overline \sigma}&\overline \sol\ar[r]^{\overline \sigma}&\overline \sol\ar[r]^{\overline \sigma}&\dots}\]
commutes whence we may extend the embedding $\alpha\colon \Z^k \to \overline \sol$ to an embedding $\wtd \alpha$ of $\Z^k[A\inv]$ as a subgroup of $\wtd \sol$.  Since $\alpha(\Z^k) $ is a discrete subgroup of $\overline \sol$, the homomorphism $\td \alpha $ embeds $\Z^k[A\inv]$ as a discrete subgroup of $\wtd \sol.$  More explicitly we have 
\[\td \alpha([(n,m)]):= [(\alpha(n), m)]\]
which is seen to be well defined by  Claim \ref{ThetaEqui}. 
As above, the embedding $\td \alpha$ of $\Z^k[A\inv]$ as a subgroup of $\wtd \sol$ defines a natural $\Z^k[A\inv]$ action on $\wtd \sol$.  
We also define a group homomorphism $\td q\colon \wtd \sol \to \sol _A$ by 
\[\td q\colon \left[\big((\xi, v) , l\big)\right] \mapsto \sigma_A^{-l}(\overline q (\xi, v))  
.\]

We enumerate properties of the above constructions.  
\begin{proposition}\label{prop:PropertiesOfSolCover}  For $\wtd \sol, \td \sigma, \td \alpha$, and $\td q$ we have
\begin{enumerate}[ label=\emph{\alph*)}, ref=\ref{prop:PropertiesOfSolCover}(\alph*)]
\item \label{PDDD} $\wtd N:= \alpha(\Z^k[A\inv])$ is a discrete subgroup  isomorphic to $\Z^k[A\inv]$; 
\item \label{2KER} $\ker(\td q) = \wtd N$, whence we have the canonical identification $\wtd \sol/\wtd N\cong \sol_A$ as topological groups;
\item\label{2LIFT} $\td q \circ \td \sigma= \sigma_A\circ \td q$;
\item\label{2STAR} $\td \sigma (  x+ \td \alpha (g) ) =  \td  \sigma(x) + \td \alpha(\tau_A (g))$ for $g\in \Z^k[A\inv]$.   
\end{enumerate}
\end{proposition}
\begin{proof}
\ref{PDDD} is clear.  
To see \ref{2KER},  let $\td q\Big( \big[\big((\xi, v), l\big) \big]\Big)= e$. 
Then we have 
\[\sigma_A^{-l}(\theta_v( \xi))=e.\]
Applying $\sigma_A^{l}$  of both sides we have 
$\theta_{v}( \xi)=\sigma_A^{l}(e)= e$ 
whence $v\in  \Z^k$.  Taking  $g = [(v,l)]$ we have 
\begin{align*}
\td\alpha(g) &=[ (\alpha(v), l)] =\big [ \big( (\vartheta_{-v}(e) , v), l\big)\big] =\big [ \big( (\xi , v), l\big)\big]
\end{align*}
hence $ \big[\big((\xi, v), l\big) \big] = \td \alpha(g)$.  Similarly one verifies that $ \td q(\td \alpha(g)) = e$ for any $g\in \Z^k[A\inv]$.  Hence \ref{2KER} holds.

To see \ref{2LIFT}, note that for any $\overline s \in \overline \sol$ and $l\in \N$ we have
\begin{align*}
\td q \circ \td \sigma( [(\overline s, l)]) = \sigma_A^{-l}(\overline q (\overline\sigma(\overline s)))
= \sigma_A^{-l}(\sigma_A (\overline q(\overline s)))= \sigma_A (\sigma_A^{-l}( \overline q(\overline s)))= \sigma_A (\td q([(\overline s,l)])).
\end{align*}
Finally for $g = [(n,m)]$ we see
\begin{align*} \td \sigma (   \td \alpha (g) ) =  [(\overline \sigma(\alpha(n)),m)]=   [(\alpha(A(n)),m)] =   \td \alpha(\tau_A (g))\end{align*} 
establishing \ref{2STAR}.  
\end{proof}
We thus have that $\td q\colon \wtd \sol\to \sol_A$ is a covering map and that $\td \sigma$ lifts $\sigma_A$.

\subsection{Metrization of $\sol_A$.}\label{sec:MetrizeSol}
We conclude this section with the construction of  a canonical metric on $\sol_A$ with respect to which  $\sol_A$ behaves (metrically) like a hyperbolic set for $\sigma_A$ when $A$ is hyperbolic.  

Firstly, let $\rho$ denote the standard metric on $\R^k$.  Given a curve $\gamma \colon [0,1]\to \sol_A$, there is a unique curve $\gamma'\colon [0,1]\to \R^n$ with $\gamma'(0) = 0$ such that  $\gamma(t) = \theta_{\gamma'(t)}( \gamma(0))$.  If  $x,y\in \sol_A$ lie in the same path component, define $\Gamma(x,y)$ to be the set of all curves $\gamma\colon[0,1]\to \sol_A$ with $\gamma(0)=x$ and $\gamma(1)=y$.  Then 
define \[\rho(x,y):= \inf_{\gamma\in \Gamma(x,y)}\rho\big(\gamma'(0), \gamma'(1)\big).\]

Secondly, for any $x,y \in \sol_A$ we define \[ \mathcal J(x,y) :=\{j\in \Z \mid p_0(\sigma_A^j (x)) \neq p_0(\sigma_A^j (y))\}\] and  
\[d_\Sigma(x,y):= \sum_{j\in \mathcal J(x,y)} 2^j.\]
Note that for any $x,y\in \sol_A $ and $v\in \R^k$ we have $d_\Sigma(\theta_v (x) , \theta_v(y)) = d_\Sigma(x,y)$ and  $d_\Sigma(\sigma_A (x) , \sigma_A(y)) = \frac{1}{2} d_\Sigma(x,y)$.

Given $x,y \in \sol_A$, let $\Xi(x,y)$ be the set of all sequences $\xi=(x_0, y_0, x_1, y_1, \dots, x_l, y_l)$ such that  there is a curve $\gamma_j\subset \sol_A $ with endpoints $x_j$ and $y_j$ for $0\le j\le l$.  
Then define \[l(\xi) := \sum_{0\le j\le l} \rho(x_j, y_j) +  \sum_{0\le j\le l-1} d_\Sigma(y_j, x_{j+1} )\] 
and \[d (x,y):= \inf _{\xi \in\Xi(x,y)} \{l(\xi)\}.\]

Now, denoting \begin{align*}\wtd \Sigma :&= \varinjlim \{\Sigma \xrightarrow{\sigma_A} \Sigma \xrightarrow{\sigma_A} \Sigma \xrightarrow{\sigma_A} \dots\} \\& = \{(x_0, x_1, \dots ) \in \sol_A\mid A^j(x_0) = [0] \mathrm {\ for \ some \ } j \in \N\}\end{align*}
we have $\wtd \sol \cong \wtd \Sigma \times \R^k$.  
Hence if $A\in \Mat(k, \Z)$ is a hyperbolic matrix, $E^+$ and $E^-$ denote the expanding and contracting subspaces,
then $ \wtd \sol \cong \wtd \Sigma\times E^+ \times E^- $.  Furthermore, if $\td d$ is the lift of the metric $d$ to $\wtd \sol$ then $\td d$ is equivalent to the product metric $\restrict {d_\Sigma} {\wtd \Sigma}\times \restrict{\rho} {E^+}\times \restrict{ \rho} {E^-}$, hence with respect to $\td d$ \[\td \sigma \colon E^+ \to E^+ \] is expanding 
and   \[\td \sigma \colon \wtd \Sigma \times  E^-\to  \wtd \Sigma\times E^-  \] is contracting.

\section{Proof of Theorem \ref{thm:main}}\label{sec:main}
Since we are only concerned with the topology of $\Lambda$ and the dynamics $\restrict f \Lambda$ in Theorem \ref{thm:main},  by Claim \ref{rem:Intertible} we assume without loss of generality that $f\colon M \to M$ is a diffeomorphism.  
To prove Theorem \ref{thm:main} we first present some preliminary observations and constructions that will enable us to build the essential dichotomy.  

\subsection{Preliminaries} \label{sec:Prelim}
Let $\Lambda$ be a compact, topologically mixing, hyperbolic attractor for a diffeomorphism $f\colon M\to M$ such that $\restrict {\dim E^u}{\Lambda} = 1$.  Let $B$ denote the basin of $\Lambda$, $\wtd B$ the universal cover of $B$, $\pi\colon \wtd B\to B$  the covering projection, and $\wtd \Lambda:= \pi\inv (\Lambda)$.  Let $G:=\pi_1(B)$ denote the fundamental group of $B$, which we identify with the group  of deck transformations for the covers $\pi\colon \wtd B\to B$ and $\restrict{\pi}{\wtd \Lambda}\colon \wtd \Lambda \to \Lambda$.  Given subsets $H\subset G$ and $X\subset B$ denote by $\orbit_H(X) := \bigcup_{g\in H} g(X)$ the \emph{orbit} of $X$ under $H$.  

Let $g$ be a Riemannian metric on $M$.  Note that for any lift $\td f$ of $f$, $\wtd \Lambda$ is a hyperbolic set under the pull-back metric $\pi^*(g)$.  For $x\in \wtd \Lambda$ and $\sigma\in \{s,u\}$ denote by $W^\sigma( x)$ the $\sigma$-man\-ifold of $x$ under the dynamics $\td f$ in the metric $\pi^*(g)$.  Note also that $W^\sigma( x)$ is the connected component of $\pi\inv\big(W^\sigma( \pi(x))\big)$ containing $x.$  In addition note that for $g\in G$ different from the identity,  $W^\sigma ({g(x)}) \cap W^\sigma( x )= \emptyset$. Given a subset $X\subset \wtd \Lambda$ we will write $W^\sigma (X):= \bigcup_{x\in X} W^\sigma(x)$.

\begin{definition}\label{defn:stabMetric}
Let $x\in \wtd B$.  Define $d^s_x$ to be  the distance function on $\stab x$ induced by restricting the metric ${\pi ^* (g)}$ to ${\stab x}$.  For simplicity we shall suppress the dependence on $x$ and simply write $d^s(x,y):=d_x^s(x,y)$ whenever $y\in \stab x$.  
\end{definition}

Note that $\wtd B$ admits a co\-dimen\-sion-1 foliation $\Fols$ by the stable manifolds of $\wtd \Lambda$.  Since $\pi_1(\wtd B)= \{1\}$, the foliation  $\Fols$ is transversely orientable.  (Indeed,   one can always  make $\Fols$ and  $\wtd B$ orientable by passing to double covers.)  
Fix a transverse orientation for $\Fols$.  Note that neither $G$ nor any lift of $f$ is assumed to preserve this transverse orientation.

Given a compact, oriented, $C^1$ curve $\gamma\subset \wtd B$  that is everywhere transverse to the foliation $\Fols$, we define a signed length $l^u(\gamma)$.  Let $\{V_i\} $ be a cover of $\Lambda$ by local $s$-product neighborhoods (see Definition \ref{def:LPC}) and let $\{\wtd V_{ij}\} = \pi \inv (\{V_i\})$  where for each $j$, the set $\wtd V_{ij}$ is homeomorphic to $V_i$.  Let $\td f$ be a lift of $f$.  Then we may find some $n>0$ so that $\gamma \subset \td f^{-n} (\bigcup_{ij} \wtd V_{ij})$.

\begin{definition}\label{defn:unstLength}
Given $\gamma$ and $n$ as above, we define the \emph{signed unstable length} $l^u(\gamma)$ as follows. We first define $l^u$ on a connected component $\gamma'$ of $\left(\gamma \cap \td f^{-n}(\wtd V_{ij})\right)$.  Let \[C = \pi \left (\td f^{n}(\gamma')\right)\] and define 
\[l^u \left(\gamma'\right) = \mathrm{sgn}(\gamma)\, e^{-nh} \, \mu^u\left( \left\{y\in W^u_{V_{i}}(z)\mid W^s_{V_{i}}(y ) \cap C
\neq \emptyset \right\}\right)\] where $z $ is any point in $V_i \cap \Lambda$,   
$\mu^u$ is as in Theorem \ref{thm:RSS}, 
\[\mathrm{sgn}(\gamma) = \begin{cases}  1, & \text{if }\gamma \text{ is positively orientated with respect to } \Fols,\\-1, & \text{if } \gamma \text{ is negatively orientated with respect to } \Fols, \end{cases}\]
and $h$ is the topological entropy of $\restrict f \Lambda$.
We may then extend $l^u$  to all of $\gamma$ additively.
\end{definition}
Theorem \ref{thm:RSS} shows that $l^u$ is well defined, independent of all choices above.  
Furthermore, given any piecewise-smooth oriented curve $\gamma\subset \wtd B$ we may partition $\gamma$ into a family of curves $\{\gamma_i\}$, each of which is everywhere tangent to $\Fols$ or everywhere transverse to $\Fols$; in the former case  we define $l^u(\gamma_i)=0$ while in the latter we use Definition \ref{defn:unstLength}.  Thus we may extend the definition of $l^u$ to all piecewise-smooth oriented curve in $\wtd B$.  Note that by Corollary \ref{cor:toRSS},   $l^u(\gamma) $ is non-zero on any curve $\gamma$ transverse to $\Fols$ and $l^u(\{x\}) = 0$. 

Now piecewise-smooth curves generate the group of piecewise-smooth simplicial 1-chains.  Thus we may extend the function $l^u$ to a piecewise-smooth simplicial 1-cochain denoted by $\alpha^u$. 

\begin{claim}\label{prop:exact}
The cochain $\alpha^u$ is closed, hence exact.
\end{claim}
\begin{proof}
A 1-cochain is closed if it is locally independent of path.  This is clear  by Theorem \ref{thm:RSS}\ref{RSS:CI}.
\end{proof}
Claim \ref{prop:exact} has the following two corollaries.  
\begin{corollary}\label{cor:sameULength}
Given $x, y \in \wtd B$ and two oriented piecewise-smooth curves $\gamma_1, \gamma_2$ with end points $x$ and $y$ then $| l^u(\gamma_1) | = | l^u(\gamma_2) |.$
\end{corollary}
\begin{proof}Changing orientation if necessary we may assume that the concatenation  $\gamma_1\cdot \gamma_2$ is a closed 1-chain.  But then we have \[0=\alpha ^u(\gamma_1\cdot \gamma_2)= l^u(\gamma_1) + l^u(\gamma_2)\] hence $l^u(\gamma_2) = - l^u(\gamma_1)$.
\end{proof}

\begin{corollary}\label{cor:intonepoint}
For each pair $x, y \in \wtd \Lambda$, the intersection $\stab x \cap \unst y$ contains at most one point. 
\end{corollary}
\begin{proof}
If not we could find a piecewise smooth 1-cycle $\gamma$ with $|\alpha^u(\gamma)|>0$, a contradiction since $\alpha ^u$ is exact.  
\end{proof}
The above corollaries motivate the following definitions.
\begin{definition}
We say a subset $V\subset \wtd \Lambda$ is a \emph{product chart} if $x, y\in V$ implies $\unst x\cap \stab y$ is non-empty and $\unst x\cap \stab y\subset V$.  
\end{definition}
\begin{definition}
For $x\in \wtd \Lambda$ and $x'\in \unst x$ let $l^u(x,x') := l^u(\gamma_{x x'})$ where $\gamma_{x x'}$ is the unique oriented curve in $\unst x$ from $x$ to $x'$.
For $x\in \wtd \Lambda$ and $L\in \R$,  let $x+_u L$ denote the unique point $x'\in \unst x $ with $l^u(x,x') = L$.
\end{definition}
\begin{definition}\label{defn:unstMetric}
Given $x,y \in \wtd B$ we define the pseudometric $d^u (x, y):= |l^u(\gamma)|$ for any piecewise smooth curve $\gamma$ with endpoints $x$ and $y$.  Furthermore we define a metric on leaves of  the foliation  $\Fols$ by $d^u(\stab x, \stab y) := d^u(x,y)$.  
\end{definition}
Note that Corollary \ref{cor:sameULength} guarantees $d^u$ is well defined, and Corollaries  \ref{cor:toRSS} and \ref{cor:toRSS2} guarantees that  the restriction of $d^u$ to $\unst x$ defines a complete metric consistent with the topology on $\unst x$.

\begin{definition}\label{defn:MetOnUC}
For $x,y \in \wtd \Lambda$ we define $\Xi(x,y)$ to be the set of sequences $\xi =(x= x_0, y_0, \dots, x_k, y_k = y)$ where $y_j\in \unst {x_j}$ for $0\le j\le k$ and $x_{j+1} \in \stab {y_j}$ for $0\le j\le k-1$.  Then define  
\[d(x,y) := \inf _{\xi \in \Xi(x,y)} \left\{ \sum _{j = 0}^k d^u(x_j, y_j)+ \sum _{j= 0}^{k-1} d^s(x_{j+1}, y_j)\right\} \]
where $d^s$ is the  distance induced by the Riemannian metric in Definition \ref{defn:stabMetric} and $d^u$ is the pseudometric constructed from the measure $\mu^u$ in Definition \ref{defn:unstMetric}.
Clearly $d$ defines a metric on $\wtd \Lambda$ consistent with the ambient topology. 
\end{definition}

\subsubsection{Global product relation} We now define a binary relation on points in $\wtd \Lambda$.
\begin{definition}[Global Product Relation]\label{def:EQrel}
For $x,y\in \wtd\Lambda$ we say $x \sim y$ if $y\in \stab x$ and \[\unst y \cap \stab{ x'} \neq \emptyset\] for all $x'\in \unst x$.  
\end{definition}
\begin{claim}
$\sim$ is an equivalence relation.  
\end{claim}
\begin{proof}
Clearly $\sim$ is reflexive.  To see that $\sim$ is symmetric suppose $x\sim y$, and that there exists some $y' \in \unst y$ such that $\stab {y'} \cap \unst x = \emptyset$.  Set $L = l^u(y, y')$ and $x' = x +_u L$.  Then $$l^u(y, \unst y \cap \stab{x'}) = L$$ 
hence $y' = \unst y \cap \stab{x'}$ contradicting the assumptions on $y'$. 
Thus $\sim$ is symmetric.  A similar argument shows that $\sim$ is transitive.  
\end{proof}
We let $[x]$ denote the equivalence class of $x$ under the relation $\sim$.  
\begin{remark}
The equivalence class $[x]$ represents the maximal subset of $\wtd \Lambda \cap \stab x$ with global product structure, that is, 
admitting the canonical homeomorphism 
\[  [x] \times \unst x\cong  \unst {[x]} \]
given by \[(y, x')\mapsto \unst y \cap \stab{x'}.\]

Furthermore we have that $W^u$ saturation of $[x]$ is $\sim$-saturated whence we have the equality 
$$\unst{[x]} = [\unst{x}]$$
and $\unst{[x]}$, with the quotient topology,  is homeomorphic to $\R$.
\end{remark}
We enumerate a number of properties of the equivalence classes of $\sim$.
\begin{claim}\label{prop:classProperties} \  
\begin{enumerate}[ label=\emph{\alph*)},  ref=\ref{prop:classProperties}(\alph*)]
	\item \label{class4} The equivalence classes are preserved under $u$-holonomy; in particular, the $\R$-action  $x\mapsto x+_u L$ on $\wtd \Lambda$ descends to a well defined $\R$ action $[x] \mapsto [x +_u L ] = [x] +_u L$. 
	\item \label{class5} Equivalence classes are preserved by the covering action of $G$ and by any lift $\td f$ of $f$.
	\item \label{class1} The equivalence classes of $\sim$ are closed, both as subsets of the stable manifolds and hence as subsets of $\wtd \Lambda$.  
	\item \label{class2} Let $C^s(x)$ denote the connected component of $\wtd \Lambda \cap \stab x$ containing $x$.  Then $C^s(x) \subset [x]$.  
	\item \label{class3}  Let $y \in \stab x $ be such that $C^s(y)$ contains  points arbitrarily close to $C^s(x)$.  Then $C^s(y) \subset [x]$.  
\end{enumerate}
\end{claim}
\begin{proof}
\ref{class4} and \ref{class5} are trivial.  

To see \ref{class1} let $x_j\to x$ in $\stab x$ where $x_j\sim x_k$ for all $j,k\in \N$.  Suppose there is some $x'\in \unst x$ so that $\stab {x'} \cap \unst{x_j} = \emptyset$ for some (hence all) $j$.  Let $C\subset \unst x $ be a compact connected set containing $x$ and $x'$.  Then there is some rectangle $V\subset \Lambda$ containing $\pi(C)$ by Claim \ref{clm:compactInside}. But then there is a product chart $\wtd V\subset \pi\inv (V)$ containing $C$ and $x_j$ for a sufficiently large $j$ contradicting the assumption that   $\stab {x'} \cap \unst{x_j} = \emptyset$ for all $j$.   Hence \ref{class1} holds.  

Fix an $L>0$. Clearly the set \[V= \{y\in \stab x \mid \stab{y+_u r} \cap \unst x\neq \emptyset \quad  \mathrm{ for \ all \ }  |r| \le L\}\]
is open in $\stab x$.  By a similar argument as above we see that $V$ is closed hence $C^s(x) \subset V$. Since $L$ was arbitrary \ref{class2} follows.  


For \ref{class3}, let $y$ satisfy the hypotheses and suppose $y'\in \unst y$ is such that $\stab {y'} \cap \unst x = \emptyset$ and let $L = l^u(y,y')$.  Since $\Lambda$ is compact we may find some $\delta>0$ so that for every $z\in \Lambda$ there is some rectangle $V(z, L)$ containing both the sets $\locStab[\delta] z\cap \Lambda$ and $\pi\left(\{\td z +_u r \mid r\le |L|\}\right)$ where $\td z$ is some lift of $z $ to $\wtd \Lambda$.  By assumption we may find a $w\in C^s(y)$ and $x'\in C^s(x)$  so that $d^s(w,x') < \delta$; setting $w'= w+_u L$  we may find a product chart containing $w,w'$, and $x'$, hence $\stab{w'}\cap \unst {x'} \neq \emptyset$.  
By \ref{class4} $w'\in [y'] $ and by \ref{class2} $x'\in [x]$ whence 
\[\stab {y'} \cap \unst x = \stab{w'} \cap \unst x = \stab{x'+_u L} \cap \unst x \neq \emptyset\]
a contradiction.
\end{proof}

We now define a metric on the quotient $\wtd \Lambda/{\sim}$ and study its induced topology.  
\subsubsection{Metrization of $\wtd \Lambda/{\sim}$.}
Denote by $\wtd \Omega$ the set of equivalence classes of $\wtd \Lambda$ under the equivalence relation $\sim$ 
and by $\Omega ^s([x]) $ the set of equivalence classes of $\wtd \Lambda \cap \stab x$ under $\sim$.  We introduce metric topologies on $\wtd \Omega$ and $\wtd \Omega^s$.  Note that the pseudometric $d^u $ on $\wtd \Lambda$ descends to a pseudometric on $\wtd \Omega$; that is, given two points $[x], [y]\in \wtd \Omega$  \[d^u([x],[y]):= d^u(x,y)\]  is well defined. 
We define a metric on each $\Omega^s([x])$ as follows.
\begin{definition}\label{def:StabMetric} Given $[x]\in \wtd \Omega$ and $[y] \in \Omega^s([x])$ let 
	\begin{align*}\label{eq:rs}
	r^s([x],[y]):= \sup\{r> 0 \mid \unst y \cap \stab {x\pm_u r'}\neq \emptyset \quad \forall \  0<r'< r\}\end{align*}
and 
\[d^s_\Omega([x],[y]) = \begin{cases}\dfrac{1}{r^s([x],[y])},& r^s([x],[y]) \neq \infty, \\ 0, &r^s([x],[y]) = \infty.\end{cases}\]
\end{definition}
Note that $r^s([x],[y]) = r^s([y],[x])$, $r^s([x],[y])>0$, and that $r^s([x],[y])\neq \infty$ unless $[x]= [y]$. Furthermore,
\begin{lemma}
$d^s_\Omega([x],[y]) $ is a metric on $\Omega^s([x])$.
\end{lemma}
\begin{proof}
We clearly have $d^s_\Omega([x],[y]) = 0$ if and only if  $[x]=[y]$ and $d^s_\Omega([x],[y]) = d^s_\Omega([y],[x])$.  Thus we need only prove the triangle inequality
	\begin{align}\label{eqn:tri}d^s_\Omega([x],[y])\le d^s_\Omega([x],[z])+ d^s_\Omega([z],[y]).\end{align}
By definition of $r^s$ we have for $[y],[z]\in \Omega^s([x])$
\[r^s([x],[y]) \ge \min \{r^s([x],[z]), r^s([z],[y])\}. \]
Thus \begin{align}\label{eq:Lion}d^s_\Omega([x],[y])\le \max \{d^s_\Omega([x],[z]), d^s_\Omega([z],[y])\}\end{align}
and \eqref{eqn:tri} holds.  
\end{proof}

\begin{definition}\label{def:Metric}
Given two points $[x], [y]\in \wtd \Omega$ let $\Xi([x],[y])$ be the space of all sequences $([x_0] , [y_0],[x_1], [y_1] \dots,[x_k],[ y_k])$ in $ \wtd \Omega$ with 
\begin{enumerate}
\item $[x_0] = [x]$, and $[y_{k}] = [y]$
\item $y_{j} \in \unst {x_{j}}$ for $0\le j\le k$
\item $x_{j} \in \stab {y_{j-1}}$ for $1\le j\le k$.
\end{enumerate}
Given a $\xi = ([x_0], [y_0], \dots,[ y_{k}]) \in \Xi([x],[y])$ define 
\[l(\xi) := \sum _{j = 0}^{k} d^u( x_{j},y_{j})     +  \sum _{j = 1}^{k} d^s_\Omega ([x_{j}], [y_{j-1}])\]

and define $d_\Omega([x],[y])$ by 
\begin{align*} d_\Omega([x],[y])&:= \inf \left\{l(\xi)\mid \xi\in \Xi([x],[y])\right\}. \end{align*}
Clearly $d_\Omega$ defines a metric on $\Omega$.  
\end{definition}

\begin{corollary}
The group $G= \pi_1(B)$ acts via isometries on $(\wtd \Omega, d_\Omega)$.  Furthermore the dynamics on $\wtd \Omega$ induced by the dynamics $\td f \colon \Lambda \to \Lambda$, which we also denote by $\td f \colon \wtd \Omega\to \wtd \Omega$, acts conformally:  $d_\Omega (\td f ([x]),\td f([y])) = e^h d_\Omega ( [x],[y])$ for $[y]\in \unst {[x]}$ and 
$d_\Omega (\td f ([x]),\td f([y])) = e^{-h} d_\Omega ( [x],[y])$ for $[y]\in \Omega^s( [x])$.
\end{corollary}
\begin{proof}
The pseudometric $d^u$ is preserved under $G$.  Since $d^s_\Omega$ is defined via $d^u$, it is also preserved.  Furthermore, we have that $d^u$ transforms according to Theorem \ref{thm:RSS}.
\end{proof}
Note that since $G$ acts via invertible isometries on $(\wtd \Omega, d_\Omega)$, it acts via homeomorphisms on $(\wtd \Omega, d_\Omega)$ despite the fact that the metric topology may not coincide with the quotient topology.  

\begin{lemma}\label{lem:SomeProps}
For $\wtd \Omega$ and $\Omega^s([x])$ we have
\begin{enumerate} [ label=\emph{\alph*)}, ref=\ref{lem:SomeProps}(\alph*)] 
	\item \label{top0a} the topology on $ \Omega^s([x])$ induced by the metric $d^s_\Omega$ is weaker than the quotient topology on $\Omega^s([x])$ inherited as the quotient ${ \Omega^s([x]) = (\wtd \Lambda\cap \stab x)/{\sim}}$;
	\item \label{top0b} the topology on $\wtd \Omega$ induced by the metric $d_\Omega$ is weaker than the quotient topology on $\wtd \Omega$ inherited as the quotient ${\wtd \Omega = \wtd \Lambda/{\sim}}$.  
	\end{enumerate}
Furthermore for  $\wtd \Omega$ and $\Omega^s([x])$ endowed with their metric topologies
\begin{enumerate}[ label=\emph{\alph*)},  ref=\ref{lem:SomeProps}(\alph*)] 
\addtocounter{enumi}{2}
	\item\label{contQuo} the quotient map $\wtd \Lambda \to \wtd \Omega$ is continuous;
	\item \label{top1} $\wtd \Omega$ and $\Omega^s([x])$ are  Hausdorff;
	\item \label{top3} either $\Omega^s([x])$ is  perfect for all $[x]\in \wtd \Omega$ or is a singleton for all $[x]\in \wtd  \Omega$. 
\end{enumerate}
\end{lemma}
\begin{proof}
To see \ref{top0a}, fix a $t>0$ and let $U:=B_{d^s_\Omega}([x], t)$.  Then 
\[U = \left\{y\in \stab x \cap \wtd \Lambda \mid \unst y \cap \stab{x\pm_u r'}\neq \emptyset \quad \forall \ 0<r'< \frac{1}{t}    \right\}.\]
Clearly $U$ is open as a subset of $\stab x\cap \wtd \Lambda$ since for any $y\in U$ we may find an open  product chart containing $y$ and $y\pm_u\frac{1}t$.  \ref{top0b} then follows from \ref{top0a}, and \ref{contQuo} follows from \ref{top0b}.
\ref{top1} follows since the topologies are metric. 

To see \ref{top3} assume $\Omega^s([x])$ is not perfect, hence  contains an isolated point $[z]$.   Then for all $z'\in \unst z$ we have $[z']$ is isolated in $\Omega^s([z'])$.  Periodic points are dense in $\Lambda$, hence we may find some periodic $q\in \Lambda$ so that $[\td q] \in \unst {[z]}$ for some lift $\td q $ of $q$ in $\wtd \Lambda$.  Furthermore, since $\stab q\cap \Lambda$ is dense in $\Lambda$ we may assume that if there is any $[z']\in \unst {[z]}$ so that $\Omega^s([z']) \neq [z']$ then $\Omega^s([\td q]) \neq [\td q]$.  Passing to an iterate of $f$ and choosing an appropriate lift $\td f\colon \wtd  \Lambda \to \wtd \Lambda$ of $f$ we may assume that $\td f ([q]) = [q]$.  Since $\td f$ is conformally contracting on $\Omega^s([\td q])$, the assumption $\Omega^s([\td q]) \neq [\td q]$ contradicts that $[\td q]$ is isolated in $\Omega^s([\td q])$.  Thus we conclude for every $[z']\in \unst {[z]}$ that $\Omega^s([z']) = [z']$.  Furthermore, we must have  $\wtd \Omega = \orbit_{G}\big( \unst{[z]}\big)$, hence we have $\Omega^s([x])$ is a singleton for every $[x]\in \wtd  \Omega$.  
\end{proof}

The proof of Theorem \ref{thm:main} will follow from considering two cases.  In the first case we will assume that  $G$ acts properly discontinuously on $\wtd \Omega$ and deduce that $\Lambda$ is expanding.    We will then show that if $G$ fails to act properly discontinuously on $\wtd \Omega$ then  $\wtd \Lambda$ has a product structure which will be used to obtain a conjugacy between $\restrict f \Lambda$ and an automorphism of a toral solen\-oid.

\subsection{Case 1: $G$ acts properly discontinuously on $\wtd \Omega$}
The goal of this section is to prove the following proposition.  
\begin{proposition}\label{prop:PDaction}
Suppose $G$ acts properly discontinuously on $\wtd \Omega$.  Then $\Lambda$ is expanding.  
\end{proposition}  

For $\td x\in \wtd \Lambda$ denote by $C^s(\td x)$ the connected component of $\wtd \Lambda \cap \stab {\td x}$ containing $\td x$.  
We define $r\colon\Lambda \to \R\cup \{\infty\}$ by \begin{align*}\label{eq:r} r\colon x \mapsto \sup _{\td y\in C^s(\td x)  } \{ d^s(\td x, \td y)\}\end{align*}
where $\td x$ is any lift of $x$ to $\wtd \Lambda$.  

Given  a metric space $(X, \rho)$ and a subset $Y\subset X$ we call \[\diam (Y) := \sup\{\rho(x,y)\mid x,y\in Y\}\] the \emph{diameter} of $Y$.   For any $x\in X$ and $Y\subset X$ we write \[\rho(x, Y) := \inf \{\rho(x,y)\mid y \in Y\}.\]

\begin{definition}
Let $\{A_n\}$ be a countable sequence of subsets in a metric space  $(X, \rho)$.  We define the \emph{Kuratowski limit supremum} by
\begin{align*}\limsup _{n\to \infty} A_n := \{ x \in X \mid \liminf_{n \to \infty} \rho (x, A_n) = 0\}.\end{align*}
Clearly the Kuratowski limit supremum is a closed set for any collection $\{A_n\}$.  
\end{definition}

\begin{lemma}\label{lemma:connectedlimits}
Let $(X,\rho)$ be a proper metric space; that is, one for which any closed ball $\{y\in X\mid d(x,y)\le R\}$ is compact.  Let $\{A_n\}\subset X$ be a countable sequence of subsets such that
\begin{enumerate}
\item $\limsup_{n\to \infty}\diam (A_n)$ is finite;
\item  each $A_n$ is connected;
\item there exists a Cauchy sequence $\{x_n\}$ with $x_n \in A_n$.  
\end{enumerate}
Then $\limsup _{n\to \infty} A_n $ is connected.
\end{lemma}  Note that the result need not hold if assumption 3 is omitted.  
\begin{proof}
By assumption 3, we may fix $x := \lim _{n\to \infty}x_n \in \limsup_{n\to \infty} A_n $.  By assumption 1 we may find an $L$ and $N$ so that for all $n\ge N$, the inclusion $A_n \subset B(x, L)$ holds.  
Suppose $\limsup _{n\to \infty} A_n$ is disconnected.  Let $N_1, N_2\subset B(x, L)$ be two disjoint open sets such that  $x\in N_1$,  $\big(\limsup_{n\to \infty}  A_n\big)\subset N_1\cup N_2$ and  $\big(\limsup_{n\to \infty} A_n\big) \cap N_2\neq \emptyset$.

By assumption 3 we may find an $M$ such that for all $n\ge M$, we have $A_n \cap N_1 \neq \emptyset$.  Furthermore, we may find an infinite subsequence $\{n_j\}$ such that $A_{n_j}\cap N_2 \neq \emptyset$.  Since $A_{n_j}$ is connected we  have $A_{n_j} \cap \partial( N_1) \neq \emptyset$.   For each $j\in \N$ pick some $a_j \in A_{n_j} \cap \partial( N_1) $.  Then since $\partial( N_1 )$ is compact we may find some $y \in \partial(N_1)$ that is an accumulation point of $\{a_j\}$.  But  this implies that $y\in  \limsup_{n\to \infty} A_n$ whence $\limsup_{n\to \infty} A_n \cap \partial(N_1) \neq \emptyset$, a contradiction.  
\end{proof}

\begin{lemma}
The function $r\colon \Lambda \to \R^+ \cup \{\infty\}$  is upper semicontinuous.   
\end{lemma}
\begin{proof}
We prove the lemma for the pull-back of the function $r\colon \wtd  \Lambda \to \R^+ \cup \{\infty\}$.
Clearly the lemma holds at $x\in \wtd \Lambda$ if $r(x) = \infty$.  We assume otherwise. 

The function $r$ is clearly continuous along unstable leaves.  
Consequently, we need only show that for $\td x_i\in \locStab x$, if $\td x_i\to x$ then  $r(x) \ge \limsup _{i\to \infty}r(\td x_i)$.  Passing to a subsequence $\{x_n\} \subset\{ \td x_i\}$ we may assume that \[\lim_{n\to \infty} r(x_n) =\limsup_{i\to \infty} r(\td x_i).\] 
If  $r(x)<\infty$ but the lemma failed at $x$, we could  find $\epsilon>0$ and $K$ so that for all $n>K$ we have $r(x_n) > r(x) + \epsilon$ and $d^s(x, x_n)< \epsilon/3$.  Let $\wtd C ^s(x_n)$ denote the  connected component of $C^s(x_n) \cap \overline{B_{d^s}(x,  r(x) +\epsilon/3 )}$ containing $x_n$. (Here $B_{d^s}(x, R)$ denotes the $d^s$-ball in $\stab x$ of radius $R$.)

Let $\Xi = \limsup_{n\to \infty}\wtd C^s(x_n) $.  By Lemma \ref{lemma:connectedlimits}, $\Xi$ is connected and hence we have $\Xi \subset C^s(x)$.  On the other hand,  the assumption on $r(x_n)$ ensures that $\wtd C ^s(x_n) \cap \partial( B_{d^s}(x,  r(x) +\epsilon/3 ))\neq\emptyset$ for all $n\ge K$.  Hence $\Xi \cap  \partial( B_{d^s}(x,  r(x) +\epsilon/3 ))\neq\emptyset$, contradicting the definition of $r(x)$.  
\end{proof}

\begin{corollary} \label{cor:largeComps}
Either $r\equiv 0$ or $r \equiv \infty$.  
\end{corollary}
\begin{proof}
Suppose first that the range of $r$ does not contain $\infty$.  Then by upper semicontinuity, $r$ is globally bounded.  Let $M= \max \{r(x)\mid x \in \Lambda\}$.  By hyperbolicity of $f$ on $\Lambda$ and boundedness of $r$ we find an $m\in \N$ so that 
\[f^m (\pi(C^s(\td x))) \subset \locStab {f^m(x)}\]
(where $\td x$ is a lift of $x$), hence  $r(f^{m+1}(x)) \le \lambda r(f^m(x))$ for all $x \in \Lambda$.  On the other hand, since $f$ is a homeomorphism, we should have \[\max \{r(f^m(x))\mid x \in \Lambda\} = \max \{r(f^{m+1}(x))\mid x \in \Lambda\}.\] But then $M = \lambda M$ which implies $M = 0$.  

Now if $r(x)\neq \infty$ then $r(y) \neq \infty$ for all $y\in \unst x$.  Indeed, let $\td x$ be a lift of $x$, $\td y$ the lift of $y$ contained in $\unst {\td x}$, and $L= l^u(\td x, \td y)$.  Let $\mathcal U$ be a cover of $C^s(\td y)$.  Then for every $z\in C^s(\td y)$ there is an $\epsilon(z)>0$ so that $ \locStab[\epsilon(z)] z\subset U$ for some $U\in \mathcal U$ and the set
\[V(z) := \{  z'+_ul\mid {z'\in \locStab[\epsilon(z)] z, |l|\le |L|}\} \] is a product chart.  
But then $\{V(z)\} $ covers $C^s(\td x)$, whence we conclude that $\mathcal U$ admits a finite subcover.  

Thus  if $r(x) = \infty$ for some $x\in \Lambda$, then $r(y) = \infty$ for all $y\in \unst x$.  Since $\unst x$ is dense in $\Lambda$, the upper semicontinuity of $r$ implies $r\equiv \infty$.  
\end{proof}

We  thus establish that $\Lambda$ is expanding under the assumption that $G$ acts properly discontinuously on $\wtd \Omega$.

\begin{proof}[Proof of Proposition \ref{prop:PDaction}]
Let $\Omega= \wtd \Omega/G$ be the orbit space.  Note that since $G$ acts properly discontinuously, $\Omega$ is Hausdorff.   Denote the canonical projections by $\pi \colon\wtd \Lambda \to \Lambda$, $q\colon \wtd \Lambda \to \wtd \Omega$, $\pi'\colon\wtd \Omega \to \Omega$.  Consider the diagram
\[\xymatrix{\wtd \Lambda \ar[r]^q \ar[d]_\pi&\wtd \Omega\ar[d]^{\pi'}\\ \Lambda&\Omega}\]
Since the equivalence classes of $\sim$ are $G$-invariant, the $G$-orbit of $q(y)$ is equivalent to the $G$-orbit of $q(g(y))$ for any $g\in G$ and $y \in \wtd \Lambda$.  Thus we may find a map $q'$ so that the diagram
\[\xymatrix{\wtd \Lambda \ar[r]^q \ar[d]_\pi&\wtd \Omega\ar[d]^{\pi'}\\ \Lambda\ar[r]_{q'}&\Omega}\]
commutes.

Since $\Lambda$ is compact and $\Omega$ is Hausdorff, $q'$ is proper, whence $q$ is proper.  
Hence the equivalence classes of $\sim$ must be compact subsets of $\wtd \Lambda$.  By Claim \ref{class2} and Corollary \ref{cor:largeComps} this implies $r\equiv 0$; hence the connected components of $\Lambda \cap \stab x$ are singletons and $\Lambda$ is expanding.  
\end{proof}

\subsection{Case 2: $G$ fails to act properly discontinuously on $\wtd \Omega$}\label{sec:C2}
In the case that $G$ fails to act properly discontinuously at some point in $\wtd \Omega$, we show that  $\Lambda$ is homeomorphic to a toral solen\-oid and $\restrict f \Lambda $ is conjugate to a solen\-oidal automorphism.  
\subsubsection{Metric properties of $\wtd \Omega$.} 
We first enumerate some additional properties of the metric $d_\Omega$ and  the action of $G$  on $\wtd \Omega$.  

\begin{claim}\label{claim:Props} The following hold in the metric space $(\wtd \Omega, d_\Omega)$.
\begin{enumerate} [ label=\emph{\alph*)},  ref=\ref{claim:Props}(\alph*)]
\item \label{clm:P1} Let $d_\Omega([x],[y])<1$.  Then $\unst y \cap \stab x \neq \emptyset$.  
\item \label{clm:P1a} We have $[z_j] \to [x]$ in $(\wtd \Omega, d_\Omega)$ if and only if $d^u\Big([x], [z_j]\Big) \to 0$ and \[d^s_\Omega\Big([x], [\unst {z_j} \cap \stab x] \Big) \to 0.\] Note we have $\unst {z_j} \cap \stab x\neq \emptyset$ for sufficiently large $j$ by \ref{clm:P1}.
\item \label{clm:P2} Fix $L\in \R$.  If $g_i([x])\to [x]$ then  $g_i([x]+_u L)\to [x]+_u L$.
\end{enumerate}
\end{claim}

\begin{proof}
Fix $R>1$ so that $d_\Omega([x],[y])< \frac{1}{R}$.   Let $\xi = ([x_0], [y_0], \dots, [y_{k}]) \in \Xi([x],[y])$ be as in Definition \ref{def:Metric} with $l(\xi)<\frac{1}{R}$.  Then we clearly have $d^u(x, x_{j})< \frac{1}{R} < R$ for all $0\le j\le k$.   Since we must also have $ d^s_\Omega ([x_{j}], [y_{j-1}])<\frac{1}{R}$ for $1\le j\le k$, we inductively see that $\unst {y_{j}} \cap \stab x \neq \emptyset$, for each $0\le j\le k$ hence \ref{clm:P1} holds.  

For $[x_i]$ and $[y_i] $ as above, denote by $H([y_i])= H([x_i]):= \big[\unst {y_i} \cap \stab x\big] = \big[\unst {x_i} \cap \stab x\big] $.  We check that for each $y_i$ 
\begin{align}\label{eqn:Kale}    d^s_\Omega\big([x],  H([y_{i}]) \big)\le  \dfrac{R}{R^2-1}.\end{align} 
Indeed since $ d^s_\Omega ([x_{j}], [y_{j-1}])<\frac{1}{R}$ then 
\begin{align*} R - \frac{1}{R}\le r^s\big(H([x_{j}]),  H([y_{j-1}]) \big) \end{align*} 
from which we obtain
\begin{align*} d^s_\Omega\big(H([x_{j}]),  H([y_{j-1}]) \big)\le  \dfrac{R}{R^2-1}\end{align*}
for all $1\le j\le k$.  
Furthermore,  $ d^s_\Omega (H([x_{j}]), H([y_{j}]))= 0$ for all $0\le j\le k-1$, hence  applying \eqref{eq:Lion} recursively one obtains \eqref{eqn:Kale}.  In particular \begin{align}\label{eq:Pig}  d^s_\Omega\big([x],  H([y]) \big)\le  \dfrac{R}{R^2-1}.\end{align}
Hence, by setting  $y = z_j$ and letting $R\to \infty$ in  \eqref{eq:Pig}, we see that $d_\Omega([z_j], [x]) \to 0$ implies $d^u([z_j], [x]) \to 0$ and $$d^s_\Omega\big([x], H([z_j]) \big) \to 0.$$   
Furthermore, we clearly have $d^u([x], [z_j]) \to 0$ and  $d^s_\Omega([x], H([ {z_j}])) \to 0$ implies that $d_\Omega ([x], [z_j]) \to 0$ hence both implications in \ref{clm:P1a} follow.

Note that $g_i(x +_u L) = g_i(x) \pm_u L $ depending of whether $g_i$ preserves the transverse orientation on $\Fols$. However for $g$ such that $l^u\big([x], g([x])\big)=t$ and 
$$d^s_\Omega\big([x], [\unst {g(x)} \cap \stab x] \big) <\frac{1}{|t|}$$ $g$ can not reverse the orientation since otherwise we would have $$\stab {x+_u  t/2} \cap \stab {g(x+_u t/ 2)} = 
\stab {x+_u t /2} \cap \stab {g(x)-_u  {t}/ 2} \neq \emptyset, $$ a contradiction unless $g$ is the identity.  Thus we may assume that for $g_i$ in \ref{clm:P2}, $g_i(x +_u L) = g_i(x) +_u L $.

Now let $L\in \R$ be given.  By forgetting initial terms and invoking \ref{clm:P1} and \ref{clm:P1a},  we may assume that for all $i$ \[H([g_i(x)]):= \big[\unst {g_i(x)}\cap \stab x\big]\] is defined, and $d^s_\Omega\left([x], H([g_i(x)])\right) < \frac{1}{|L|+1}$.  Then by definition of $r^s$ we have  ${\unst {g_i(x)} \cap \stab {x+_u L} \neq \emptyset}$, hence $H([g_i(x)]) +_u L = \big[\unst {g_i(x)} \cap \stab {x+_u L}\big]$.   As above we have 
\begin{align*}
r^s\big([x],  H([g_i(x)]) \big)- L \le r^s\big([x+_u L],  H([g_i(x)])+_u L\big) 
\end{align*} 
hence 
\begin{align*}
d_\Omega^s\big([x+_u L],  H([g_i(x)])+_u L\big) \le\dfrac{d^s_\Omega \big([x],  H([g_i(x)]) \big)} {1 -  L \cdot d^s_\Omega \big([x],  H([g_i(x)])\big)} 
\end{align*} 
which by \ref{clm:P1a}, establishes \ref{clm:P2}.  
\end{proof}

Given a subset $S\subset G$ we say $S$ acts  \emph{properly discontinuously at $[x]$} if there is some open set $U\ni [x]$ so that $s(U) \cap U \neq \emptyset$ implies $s= 1$ for any $s\in S$.   Since $G$ acts freely on $\wtd \Omega$, every finite subset $S\subset G$ acts properly discontinuously at every point of $\wtd \Omega$.  

\begin{lemma}\label{lem:PDeverywhere}
Suppose a set $S\subset G$  acts properly discontinuously at one point $[x]\in \wtd \Omega$.  Then $S$ acts properly discontinuously at every point $[y]\in \wtd \Omega$.  
\end{lemma}
\begin{proof}
Let $U\subset \wtd\Omega$ be an open neighborhood of $[x]$ so that for each $s\in S$, we have $s(U) \cap U \neq \emptyset $ implies $ s = 1$.   Then $S$ acts properly discontinuously at every point $[y]\in U$.   By Claim \ref{clm:P2}, if $S$ acts properly discontinuously at $[y]$ then it acts properly discontinuously at every point in $\unst {[y]}$, hence we have that  $S$ acts properly discontinuously at every point of $\unst U $.  But then   $S$ acts properly discontinuously at every point of  $\orbit_G(\unst U) = \wtd \Omega$. 
\end{proof}

Setting $S= G$, we have the contrapositive.
\begin{corollary}\label{cor:PDeverywhere}
If $G$ fails to act properly discontinuously at one point $[x]\in \Omega$ then it fails to act properly discontinuously at every point $[y]\in \Omega$.  
\end{corollary}
Furthermore we have
\begin{corollary}\label{cor:Converge}
Let $g_i([x]) \to [x] $ in $(\wtd \Omega, d_\Omega)$ for some sequence $\{g_i\} \subset G$.  Then $g_i([y]) \to [y] $ in $(\wtd \Omega, d_\Omega)$ for any $[y]\in \wtd \Omega$.  
\end{corollary}
\begin{proof}
Set $S = \{g_i\}$.  If $g_i([y])$ fails to converge to $[y]$ then there is a  neighborhood $U$ of $[y]$ and an infinite subset $S'\subset S$ so that $s(U) \cap U = \emptyset $ for all $s\in S'$.  But then $S'$ acts properly discontinuously  at $[y]$ which by Lemma \ref{lem:PDeverywhere} implies that $S'$ acts properly discontinuously at $[x]$.  But $S'$ corresponds  to a infinite subsequence of $\{g_i\}$, contradicting that $g_i([x]) \to [x]$.\end{proof}

We now show that the above convergence happens uniformly in $[y]$.  
Define a map $\zeta\colon G\times \wtd \Omega\to [0,\infty)$ by \[\zeta\colon (g,[x]) \mapsto d_\Omega([x], g([x]) ).\]
Endowing $G$ with the discrete topology, we have that $\zeta$ is continuous.  Since the metric topology on $\wtd \Omega$ is weaker than the quotient topology, the quotient map $q\colon \wtd \Lambda \to \wtd \Omega$ induces a continuous map $q^*\zeta\colon G\times \wtd \Lambda \to [0,\infty)$.  Now \[q^*\zeta(g, x) = q^*\zeta(g, g'(x))\] for all $g,g'\in G$ hence $q^*\zeta$ induces a continuous map $\overline \zeta\colon G \times \Lambda \to [0,\infty)$.

As a result we have,
\begin{lemma}\label{lem:unf}
Assume  $g_i([x] )\to [x]$ for some $[x]\in \wtd \Omega$.   Then given an $\epsilon>0$, we may find an $N$ so that for all $i\ge N$ and $[y]\in \wtd \Omega$ we have $d_\Omega([y], g_i([y]))<\epsilon$.  
\end{lemma}
\begin{proof}
Assume the conclusion fails for some fixed $\epsilon$. Passing to an infinite subsequence, we may assume the conclusion fails for all $i\in \N$.  Let $\{y_i\} \subset \Lambda$ be such that $\overline \zeta(g_i, y_i)>\epsilon$ and, again passing to a subsequence, let $z\in \Lambda$ be a limit point of $\{y_j\}$.  Let $\td z$ be a lift of $z$ and $\{\td y_i\}$ a lift of $\{y_i\}$ such that $\td z $ a limit point of $\{\td y_i\}$.  Because the metric topology on $\wtd \Omega$ is weaker than the quotient topology, we have that $[\td z] $  is a limit point of the sequence $\{[\td y_i]\}$ with respect to the metric topology.  

Passing to a subsequence we may assume $d_\Omega([\td y_i], [\td z]) < \epsilon/4$ for all $i$ from which we obtain
\[d_\Omega([\td z], g_i([\td z]))\ge d_\Omega([\td y_i], g_i([\td y_i]))- d_\Omega([\td z], [\td y_i]) - d_\Omega( g_i([\td y_i]),g_i([\td z]) )\]
hence $d_\Omega([\td z], g_i([\td z]) )> \epsilon/2$ for all $i$, contradicting Corollary \ref{cor:Converge}.
\end{proof}

\subsubsection{Global product structure}\label{section:GPS} We now establish that when $G$ fails to act properly discontinuously, the set $\wtd \Lambda$ has a \emph{global product structure}; that is, for all $x, y \in \wtd \Lambda$ we have $\unst x \cap \stab y \neq \emptyset$.  

\begin{lemma}\label{lem:notPD}
{Let $G$ fail to act properly discontinuously  on $\wtd \Omega$.  Then $\Omega^s([x])$ is a singleton for all $x$ and $\wtd \Lambda = \unst{[x]}$ for any $x\in \wtd \Lambda$.  }
\end{lemma}

\begin{proof}
Fix some $R>1$, and choose an $R'>R$ with the property that $ \frac{R'}{(R')^2 - 1}\le \frac{1}{R}$. 

Suppose $G$ fails to act properly discontinuously at $[x]$.  Then we may find a subset $\{g_i\}\subset G$ so that 
$g_i([x]) \to [x]$ and $d_\Omega([x], g_i([x]))<1/R$ for all $i$.  Let $N= B_{d^s_\Omega}([x], \frac{1}{R})$.   As guaranteed by  Lemma  \ref{lem:unf}, we may remove initial terms of $\{g_i\}$ so that $d_\Omega([y], g_i([y])) \le \frac{1}{R'}$ for all $i$ and $[y]\in \wtd \Omega$.  

For $[y]\in N$  define 
\[H_{g_i} ([y]) := \big[\unst {g_i(y)} \cap \stab x \big].\] Then, as in \eqref{eq:Pig}  we have 
\[d^s_\Omega([y], H_{g_i}([y])) \le \dfrac{R'}{(R')^2 - 1}\le \frac{1}{R}.\]
But then by \eqref{eq:Lion} we  have 
\[d^s_\Omega([x], H_{g_i}([y]))  \le \max \{d^s_\Omega([x], [y]), d^s_\Omega([y], H_{g_i}([y]))\}\le \frac{1}{R}\]
hence $H_{g_i} (N )\subset N$.  Furthermore $H_{g_i\inv}$ is defined on $N$ and by the same argument as above 
$H_{g_i\inv} (N )\subset N$.  Hence $H_{g_i}( N) = N$.  In particular, setting $L=l^u(x, g_i(x))$ we have
\[g_i(N) = N+_u L.\]

Set 
\[D = \{ y+_u l\mid {|l|< R, [y] \in N}\} .\]
Since $d^u([x], g_i([x]))<\frac {1}{R}< R$, we have ${g_i}( N) \subset D$ and ${g_i\inv}( N) \subset D$.  Inductively, we see that for any $k$ and \[[x], [y] \in \bigcup _{|j|\le k} {g_i^j}( D) \]
that $\unst x \cap \stab y \neq \emptyset$, hence $\bigcup _{|j|\le k} {g_i^j}( D)$ is a product chart.  In particular, we have equality between product charts
$\bigcup _{j\in \Z} {g_i^j}( D)=  \unst{ [x]}$, thus showing that $[x]$ is isolated in $\Omega^s([x])$.  
By Lemma \ref{top3} we see that $\Omega^s([x]) = [x]$ for all $[x]\in \wtd \Omega$.

Considering $[x]$ as a subset of $\wtd \Lambda$ we have $\wtd \Lambda = \orbit_G(\unst {[x]})$ and $$\wtd B = \orbit_G(\stab{\unst {[x]}}).$$  If $g\in G$ is such that $g(\unst {[x]}) \neq \unst {[x]} $ then $$g(\stab{\unst {[x]}}) \cap (\stab{\unst {[x]}}) = \emptyset.$$  Thus we must have $\wtd \Lambda = \unst {[x]}$ since otherwise $\wtd B$ would not be  connected.  
\end{proof}

The following is immediate from Lemma \ref{lem:notPD}.

\begin{corollary}\label{cor:GPS}
When $G$ fails to act properly discontinuously on $\wtd \Omega$ then $\wtd \Lambda$ admits a {global product structure}.
\end{corollary}

We now shift out attention back to $\wtd \Lambda$, under the assumption that $\wtd \Lambda$ admits a global product structure.  Our objective  is to prove the following.
\begin{proposition}\label{prop:MainAlg}
Assume $\wtd \Lambda$ has a global product structure.  Then $f\colon \Lambda \to \Lambda$ is conjugate to a leaf-wise hyperbolic automorphism  of a toral solen\-oid (see Section \ref{sec:AS}).
\end{proposition}
To prove Proposition \ref{prop:MainAlg} we need the following technical result.  We note that the proof technique for Proposition \ref{prop:MainAlg}, including Lemma \ref{lem:Shadowing}, are adapted from \cite{MR1836432}.  

\subsubsection{Global shadowing lemma}
Let $(\Upsilon, \tau)$ be a metrizable topological space.  For a fixed $k$ let $\rho$ be the standard metric on $\R^k$. 
  Furthermore let
$\{d_x\}_{x\in \R^k}$ be a family of complete metrics on $\Upsilon$ (each inducing the topology $\tau$) 
such that
\begin{enumerate}
\item $d_x$-balls in $\Upsilon$ are precomact for all $x\in \R^k$
\item \label{cont} the induced map $\R^k\times \Upsilon\times \Upsilon\to \R$ given by \[(x, \xi, \eta) \mapsto d_x(\xi,\eta)\] is continuous.
\end{enumerate}
Let $\Omega = \R^k\times \Upsilon$  with projections $\pi_1\colon \Omega \to \R^k$ and $\pi_2 \colon \Omega\to \Upsilon$.   Given $x,y \in \Omega$ let $\Xi(x,y)$ be the set of sequences $\{x_0, y_0, \dots, x_k, y_{k}\} $ such that 
\begin{enumerate}
\item $x= x_0$ and $y = y_{k}$;
\item $\pi_1 (x_{j}) = \pi_1 (y_{j}) $ for all $0\le j\le k$;
\item $\pi_2 (x_{j}) = \pi_2 (y_{j-1}) $ for all $1\le j\le k$.
\end{enumerate}
Given an $\xi\in \Xi(x,y)$ define 

\[l(\xi) := \sum_{j= 0}^k d_{\pi_1(x_{j})}\big( \pi_2(x_{j}), \pi_2(y_{j})\big) + \sum_{j= 1}^k\rho\big( \pi_1(x_{j}), \pi_1(y_{j-1})\big)\]
and define \begin{align}\label{eq:metricdef}  d(x,y) := \inf_{\xi \in \Xi(x,y)} \{l(\xi)\}.\end{align}  Clearly  $d$ defines a metric on $\Omega$; furthermore, the  continuity of the function $(x, \xi, \eta) \mapsto d_x(\xi,\eta)$  guarantees that the metric topology is consistent with the product topology of $\R^k \times \Upsilon$.  

Given a metric space $(X,d)$,  a homeomorphism $f\colon X\to X$ is called \emph{expanding} if there is some $\mu>1$ so that for all $x,y\in X$, $d(f(x),f(y))\ge \mu d(x,y)$.  
A sequence $\{x_j\}_{j\in \Z}\subset (X, d)$ is called an \emph{$L$-pseudo orbit} for $f$ if $d(f(x_j), x_{j+1})\le L$ for all $j$. Given an $L$-pseudo orbit $\{x_j\}$ we say a point $x\in X$ \emph{shadows} $\{x_j\}$ if there is some $\delta$ so that $d(f^j(x), x_j)\le \delta$ for all $j$.

\begin{lemma}[Global Shadowing]\label{lem:Shadowing}
Let $h\colon  \Omega\to \Omega$ be a product homeomorphism $h\colon (x, \xi) \mapsto (h_1(x), h_2(\xi))$.  Assume that $h_1\colon \R^k\to \R^k$ is expanding with respect to the metric $\rho$.  Furthermore assume that $h_2\colon \Upsilon\to \Upsilon$ is \emph{asymptotically exponentially contracting on bounded sets} with respect to each metric $d_x$; that is, given an $R>0$, $\xi \in \Upsilon$, and $x\in \R^k$ there are  $c>0$ and $\lambda<1$, depending continuously on $(x,\xi)\in \Omega$, so that if $d_x(\xi, \zeta) \le R$ then 
\[d_{h_1^j(x)}(h_2^j(\xi), h_2^j(\zeta)) \le c \lambda ^j d_x(\xi, \zeta)\] 
for all $j\ge 0$.
Additionally assume that $\Omega$ admits a properly discontinuous action by a subgroup $G$ of the group of isometries of $(\Omega,d)$ such that $h$ preserves $G$-orbits, and the quotient $\Omega/G$ is compact. 
We have the following.
\begin{enumerate}[ label=\emph{\alph*)},  ref=\ref{lem:Shadowing}(\alph*)]
\item \label{unfHol} Given a $C>0$ there is a $K>0$ so that if  $d(x,y)\le C$ for any $x, y \in \Omega$, then 
\[d_{\pi_1(y)} \big(\pi_2(y), \pi_2(x) \big)\le K.\] 

\item \label{unfContract}Given an $R>0$ there are $c>0$ and $\lambda<1$ (depending only on $R$) so that for any $\xi,\zeta \in \Upsilon$  and $x\in \R^k$,   $d_{x} (\xi, \zeta) \le R$ implies \[d_{h_1^j(x)}\big(h_2^j(\xi), h_2^j(\zeta)\big) \le c \lambda ^j d_x(\xi, \zeta).\] 

\item \label{shadow} Given any $L$-pseudo orbit $\{x_j\} \subset \Omega$ in the metric $d$, there is a point $x\in \Omega$ that shadows $\{x_j\}$.  
\item \label{uniqueCont} Fix $C> 0$.  Then there is a sequence $\{ \epsilon_N\}$ with $\epsilon_N\to 0$ as $N \to \infty$ such that $$d(h^j(x), h^j(y))\le C \mathrm{\ \ for \ \ }|j|\le N$$ implies $d(x,y) \le \epsilon_N$.
\end{enumerate}
\end{lemma}
\begin{refrem}\label{rem:UseLemma} Note that Lemma \ref{lem:Shadowing} applies to $\wtd \Lambda$ in the case that $\wtd \Lambda$ has global product structure by choosing an $x\in \wtd \Lambda$ and taking $\Upsilon = \wtd \Lambda \cap \stab x$, $ k =1$, $\rho = d^u$ in Definition \ref{defn:unstMetric}, and $d_x$ the metrics $d^s_x$ in Definition \ref{defn:stabMetric}. Then the metric in Definition \ref{defn:MetOnUC} corresponds to \eqref{eq:metricdef}. 

Furthermore, assuming the linear map $A\colon \R^n\to \R^n$ is hyperbolic, we have that the cover $\wtd \sol$ constructed in Section \ref{sec:AS} and endowed with the metric $\td d$ constructed in  Section \ref{sec:MetrizeSol} satisfies the hypotheses of Lemma \ref{lem:Shadowing} with $\Upsilon = E^-\times \wtd \Sigma$ and $k = \dim E^u$. \end{refrem}

\begin{proof}[Proof of Lemma \ref{unfHol} and \ref{unfContract}]
The hypotheses of Lemma \ref{lem:Shadowing}  guarantee that the numbers $K$, $c$, and $\lambda$ can be chosen pointwise on $\Omega$.  
Since the group $G$ acts via isometries, we may assume they are constant on $G$-orbits.  Since the quotient $\Omega/G$ is compact, we may choose uniform $K$, $c$, and $\lambda$.   
\end{proof}

\begin{proof}[Proof of Lemma \ref{shadow}]
Because $h_1$ is an expanding homeomorphism, there is a unique fixed point $p\in \R^k$.  If $\{x_j\}$ is an $L$-pseudo orbit for $h$ in the metric $d$ then ${\pi_1(x_j)}$ is an $L$-pseudo orbit for $h_1$ in the metric $\rho$; that is, $\rho (\pi_1(h(x_j)), \pi_1(x_j+1))\le L$.  But then $$\rho (\pi_1(h^{-j}(x_j)),  \pi_1(h^{-j-1}(x_j+1)))\le \mu^{-j-1} L$$ from which we see that the sequence $$\{\pi_1(x_0), \pi_1(h\inv (x_1)), \dots, \pi_1(h^{-j}( x_j)), \dots\}$$ is Cauchy.  Set \[y = \lim_{j \to \infty} \{\pi_1(x_0), \pi_1(h\inv (x_1)), \dots, \pi_1(h^{-j} (x_j)), \dots\}.\]

Let $g\colon \R\to \R$ be the map $g\colon x \mapsto \mu\inv (x+L)$.  Then $g$ is contracting, hence has a unique fixed point.  In particular the sequence $\{x, g(x), g^2(x), \dots\}$ is bounded for any $x$.  Set
\[R:= \sup\{ \rho (\pi_1(x_0), p),  g( \rho (\pi_1(x_0), p)) , g^2( \rho (\pi_1(x_0), p)), \dots\} < \infty .\]  Then  for every $j\le 0$ we have \[\rho (\pi_1(x_j), p) \le R < \infty.\]

Taking $C =  R+ L + \mu R$ we have that $d\Big(\big(p,\pi_2(x_j)\big), \big(p,\pi_2(h(x_{j-1}))\big)\Big) \le C$ for all $j\le0$, whence from  \ref{unfHol} we may find a $K<\infty$ so that $d_p(\pi_2(x_j), \pi_2(h(x_{j-1}))) \le K < \infty$ for all $j\le0$.  
Thus from \ref{unfContract} the sequence \[\{\pi_2(x_0), \pi_2(h(x_{-1})), \pi_2(h^2(x_{-2})),\dots , \pi_2(h^j(x_{-j})), \dots \}\] is Cauchy in the metric $d_p$.  Let \[z= \lim_{j\to \infty} \{\pi_2(x_0), \pi_2(h(x_{-1})), \pi_2(h^2(x_{-2})),\dots , \pi_2(h^j(x_{-j})), \dots \}.\]  Then one easily verifies that $x := (y,z)$ shadows the $L$-pseudo orbit $\{x_j\}$.  
\end{proof}

\begin{proof}[Proof of Lemma \ref{uniqueCont}]
For $C$, let $K$ be as in \ref{unfHol} and let  $c$, $\lambda$ be as in \ref{unfContract} with $R= K$.  
Fix $x, y\in \Omega$ so that $d(h^j(x), h^j(y))\le C$ for $|j|\le N$.  Then we have \[\rho(\pi_1(h^j(x)),\pi_1( h^j(y)))\le C\] for $j\le N$, hence  \[\rho(\pi_1(x),\pi_1(y)) \le \mu^{-N} C.\]  
Furthermore,

\[d_{h_1^{-N}\big(\pi_1(y)\big)} \bigg(h_2^{-N}\big(\pi_2(y)\big), h_2^{-N}\big(\pi_2(x)\big)\bigg)\le K\] 
hence 
\[d_y(\pi_2(y), \pi_2(x)) \le cK\lambda^N\]
from which we conclude that \[d(x,y)\le  cK\lambda^N+  \mu^{-N} C\]
and the conclusion follows with $\epsilon_N := cK\lambda^N+  \mu^{-N} C$.  
\end{proof}

\subsubsection{Proof of Proposition \ref{prop:MainAlg}}  
We now return to the proof of Proposition \ref{prop:MainAlg}.  
We have the following observation.  
\begin{lemma} When $\wtd \Lambda$ has global product structure 
the covering group $G:=\pi_1(B)$
is torsion-free abelian.  
\end{lemma}
\begin{proof}
Fix any $x\in \wtd \Lambda$.  Since $\wtd \Lambda$ has global product structure, we may   canonically identify $\wtd \Lambda$ with $\unst x \times \left(\wtd \Lambda \cap \stab x\right)$. 
Let $\sim_s$ be the equivalence relation ${z\sim_s y}$ if $z\in \stab y$.  Then we have a canonical identification of $\unst x$ with $\Lambda /{\sim}_s$ which induces a $G$-action  on $\unst x$.  By Theorem \ref{thm:RSS} and the construction of the pseudometric $d^u$ on $\wtd \Lambda$, we have that $G$ acts on $(\unst x, d^u)$ via isometries.  Furthermore the isometries are  orientation-preserving since otherwise there would be a non-identity $g\in G$ and $y \in \unst x$ with $\stab y = \stab{g(y)}$.  
Hence we naturally identify $G$ with a subgroup of the orientation-preserving isometries of $\R$ and the result follows. 
\end{proof}

Note that $G$ need not be finitely generated.  However, we can represent $G$ as the limit of a directed system of finitely generated, torsion-free abelian groups as follows.
Let $\{\Rec_j\}$ be a Markov partition of $\Lambda$ and let $\{\wtd \Rec_{j,\alpha}\}_{\alpha \in G}$ be a lift of the Markov partition where each $\wtd \Rec_{j,\alpha}$ is homeomorphic to $\Rec_j$.  
Fix an $x\in \wtd \Lambda$.  Recall that $\unst {\pi(x)}$ is dense in $\Lambda$.  For each $j$, distinguish a $\overline{ \Rec_j} \in \{\wtd \Rec_{j, \alpha}\}$ so that \begin{align}\label{eq:Koo} \unst x \cap \int(\overline{\Rec_j} )\neq \emptyset.  \end{align}  Set $D = \bigcup_j \overline{\Rec_j}$.  Then $D$ is a fundamental domain for the covering $\wtd \Lambda \to \Lambda$.  

Let $H\subset G$ be the subgroup generated by $\{\alpha \in G \mid \alpha (D) \cap D \neq \emptyset\}$.  Since $D$ is compact and $G$ acts discontinuously,  $H$ is finitely generated.  Let $N := \orbit_H(D)$.  Note  $N$ is clopen in $\wtd \Lambda$ and  $H=\{\alpha\in G \mid  \alpha(N) = N\}$.
\begin{claim}\label{clm:poo}
There is a lift $\td f$ of $f$ so that $\td f(N) \subset N$.  Indeed $f_*H \subset H$, where $f_*\colon G\to G$ is the automorphism induced by the diffeomorphism $f\colon B\to B$.
\end{claim}
\begin{proof}
Since $D$ is a lift of a Markov partition, by the definition of $H$ if $y\in N$ then $\unst y \subset N$.    Choose a lift $\td f $ of $f\colon \Lambda \to \Lambda$ so that \begin{align}\td f (x) \cap N \neq \emptyset \label{eqn:contain}\end{align}
where  $x$ is as chosen above.

Now $\td f(N) = \orbit_{f_*H} (\td f(D))$.  We note that $f_* H $ is the subgroup of $G$ generated by the set $\mathcal A := \{\alpha \in G \mid \alpha (\td f(D)) \cap \td f(D)\neq \emptyset\}$.  By \eqref{eq:Koo} and \eqref{eqn:contain} we have that $\int(\td f(\overline {\Rec_j}) )\cap  N\neq \emptyset$, hence by the Markov property (Definition \ref{def:markov}(3))  we have  $\td f(\overline {\Rec_j}) \subset  N$ for each $j$.  In particular $\td f (D) \subset N$. Hence we conclude that $\mathcal  A \subset H$, $f_* H \subset H$ and $\td f(N) \subset N$.  \end{proof}

Note that $N$ is a covering of $\Lambda$ with covering group $H$.  Also, for any $y\in \wtd \Lambda$ and $\td f$ as in Claim \ref{clm:poo} there is some $m$ so that $\td f^m (y) \in N$.    
Consequently, we may reconstruct $\wtd \Lambda$ and the covering group $G$ as  limits of the directed systems 
\[\Lambda\cong\varinjlim \left\{ N \xrightarrow{\td f} N \xrightarrow{\td f} N \xrightarrow{\td f} \dots \right\}\]
and 
\[G\cong \varinjlim \left\{ H \xrightarrow{f_*} H \xrightarrow{f_*} H \xrightarrow{f_*}\dots \right\}.\]

Fix an isomorphism $\Phi\colon H \to \Z^k$ and let $A\colon \Z^k \to \Z^k$ be the endomorphism $\Phi\circ f_* \circ \Phi\inv$.  Considering $\Z^k$ as embedded in $\R^k$,  $A\colon \Z^k\to \Z^k$ induces   a linear automorphism  on $\R^k$ and a surjective endomorphism on the quotient $\T^k= \R^k/\Z^k$, also denoted by $A$.  Let $\sol_A$ and $\wtd \sol$ be the solen\-oid and its cover constructed  in Section \ref{sec:AS}, and let $\sigma_A$ and $\td \sigma$ be the respective shift automorphisms.

Fix an identification $G =  \varinjlim (H, f_*)$ whence $f_*\colon G \to G$ is identified with the shift map $\tau_{\restrict{f_* }H}$.  
We have that the diagram
\[\xymatrix{H\ar[d]_\Phi \ar[r]^{f_*}&H\ar[d]^{\Phi}\\ \Z^k\ar[r]_A&\Z^k}\]
commutes, hence 
\[\xymatrix{H\ar[d]^\Phi \ar[r]^{f_*}&H\ar[d]^{\Phi} \ar[r]^{f_*}&H\ar[d]^{\Phi}\ar[r]^{f_*}&H\ar[d]^{\Phi}\ar[r]^{f_*}&\dots
\\ \Z^k\ar[r]^A&\Z^k\ar[r]^A&\Z^k\ar[r]^A&\Z^k\ar[r]^A&\dots}\]
induces an isomorphism $\wtd \Phi\colon G \to \Z^k[A\inv]$. 
Furthermore, we have $\wtd \Phi \circ f_*\circ \wtd \Phi \inv = \tau_ A$ where $\tau_ A$ is as constructed in Section \ref{sec:Cover}.

\begin{proof}[Proof of Proposition \ref{prop:MainAlg}]
Fix a lift $\td f$ of $f$.  Let $\sol _A, \sigma_A, \wtd \sol$, and $\td \sigma$ be as above.  Let $D$ be a fundamental domain for the cover $\wtd \sol\to \sol_A$; note that $D$ will be compact.  
Let $\wtd \Phi$ be the isomorphism between $G$ and $\Z^k[A\inv]$ above.  
We let $\Z^k[A\inv]$ act by addition on $\wtd \sol$ via the action $\td \alpha$ in Section \ref{sec:Cover}.

Given a $\xi \in \wtd \sol$ we may find a sequence $\{\alpha_j\} \subset \Z^k[A\inv]$ so that \begin{align}\label{eqn:Poo}\td \sigma ^j(\xi) \in D+ \alpha_j.\end{align}  
\begin{claim}
There is an $L$ so that for any $x\in \wtd \Lambda$, $\xi\in \wtd \sol$, and a sequence $\{\alpha_j\}$ satisfying \eqref{eqn:Poo},  the sequence $\{(\wtd \Phi\inv(\alpha_j))(\td f ^j( x))\}$ is an $L$-pseudo orbit.
\end{claim}
\begin{proof}
Let \[\mathcal A =\left\{a \in  \Z^k[A\inv] \mid  (D +a )\cap \wtd \sigma (D ) \neq \emptyset \right\}.\]
By Proposition \ref{2STAR} if $\xi\in D+ \alpha$ then $\wtd \sigma (\xi) \in (D+ a) +  \tau_ A (\alpha) $ for some $a\in \mathcal A$.  In particular, $\alpha_{j+1} = \tau_A\alpha_j + a$ for some $a\in \mathcal A$.   We set
\[L = \sup \{d(\td y, (\wtd \Phi\inv (a))(\td y))\mid {y\in \Lambda, a\in \mathcal A} \}\] where $\td y$ is an arbitrary lift of $y$ to $\wtd \Lambda$.  Finiteness of $\mathcal A$ guarantees $L<\infty$.  
Hence \begin{align*}
d\Big(  \td f\big( (\wtd \Phi& \inv (\alpha_j)) (\td f^j(x)) \big),  (\wtd \Phi \inv (\alpha_{j+1})) (\td f^{j+1}(x))     \Big)\\ 
&=d\left(   (f_* (\wtd \Phi \inv (\alpha_j))) (\td f^{j+1}(x)),  (\wtd \Phi \inv (\alpha_{j+1})) (\td f^{j+1}(x))     \right) \\ 
&=d\left(    (\wtd \Phi \inv (\tau_A(\alpha_j))) (\td f^{j+1}(x)),  (\wtd \Phi \inv (\alpha_{j+1})) (\td f^{j+1}(x))     \right)\\
&\le \max_{a\in \mathcal A} \left\{d\left(    (\wtd \Phi \inv (\tau_A( \alpha_j))) (\td f^{j+1}(x)),  (\wtd \Phi \inv (\tau_A(\alpha_{j} )+ a)) (\td f^{j+1}(x))     \right)\right\}\\
&\le L.
\end{align*}
Hence the claim holds.
\end{proof}

We define a map $\Psi \colon \wtd \sol \to \wtd \Lambda$ as follows.  Fix a $p \in \wtd \Lambda$.  Given $\xi \in \wtd \sol$ choose a sequence $\{\alpha_j\} \subset \Z^k[A\inv]$ satisfying \eqref{eqn:Poo}.  Then define  $\Psi( \xi)$ to be the unique point $x$  in $ \wtd \Lambda $ that shadows the $L$-pseudo orbit $\{(\wtd \Phi\inv(\alpha_j))(\td f ^j (p))\}$.  Note that Lemma \ref{shadow} guarantees the point $x$ exists, whereas Lemma \ref{uniqueCont} guarantees that the point $x$ is unique.  Furthermore, Lemma \ref{uniqueCont} guarantees that $\Psi\colon \wtd \sol \to \wtd \Lambda$ is continuous.  

\begin{claim}
$\Psi$ is proper.
\end{claim}
\begin{proof}
For $\xi\in \wtd \sol$ and $\alpha \in \Z^k[A\inv]$ we clearly have $\Psi(\xi+\alpha) = (\wtd \Phi \inv (\alpha))(\Psi(\xi))$, hence the map
$\Psi\colon \wtd \sol \to \wtd \Lambda$ descends to a continuous map $h\colon\sol_A \to \Lambda$.  
\end{proof}

Since $\Psi$ is proper, we have that $A$ is hyperbolic.  Thus as in Remark  \ref{rem:UseLemma}, Lemma \ref{lem:Shadowing} applies to $\wtd \sol$.  Hence, given a fundamental domain $D\subset \wtd \Lambda$, and  $x\in \wtd \Lambda$, we choose $\{g_i\}$  so that $\td f^i(x)\in g_i(D)$.  Then as above we define $\Psi'(x)$ to be the unique point $\xi\in \wtd \sol$ so that $\xi$ shadows the pseudo orbit
\[\{\wtd \Phi(g_i)(e)\}\]
where $e$ is the identity in $\wtd \sol$.  We thus obtain a map $\Psi'\colon \wtd \Lambda \to \wtd \sol$.  

One easily verifies 
\begin{enumerate}
\item $\Psi$ and $\Psi'$ are inverses;  
\item the diagram
\[\xymatrix{\wtd \sol \ar[r]^{\td \sigma} \ar[d]_{\Psi}&\wtd \sol \ar[d]^{\Psi}\\ \wtd \Lambda\ar[r]_{\td f}&\wtd \Lambda}\]
commutes;
\item $\Psi$ and $\Psi'$ intertwine the covering actions of $G$ and $\Z^k[ A\inv]$, that is,
\begin{itemize}
\item[---] $\Psi(\alpha (\xi)) = (\wtd \Phi \inv (\alpha))(\Psi (\xi))$ for all $\alpha \in \Z^k[ A\inv], \xi \in \wtd \sol$;
\item[---] $\Psi'(g (x)) =\Psi' (x) +  \wtd \Phi (g)$ for all $g \in G, x \in \wtd \Lambda$.
\end{itemize}
\end{enumerate}
Thus the homeomorphism $\Psi\colon \wtd \sol \to \wtd \Lambda$ induces a homeomorphism $h\colon \sol_A\to \Lambda $ such that the diagram 
\[\xymatrix{\sol_A \ar[r]^{ \sigma_A} \ar[d]_{h}& \sol_A \ar[d]^{h}\\ \Lambda\ar[r]_{ f}& \Lambda}\]
commutes. 
\end{proof}

\subsection{Proof of Theorem \ref{thm:main}}
\begin{proof}[Proof of Theorem \ref{thm:main}]
By Proposition \ref{prop:PDaction}, Corollary \ref{cor:GPS}, and Proposition \ref{prop:MainAlg}  if $\Lambda$ is not an expanding attractor, then $\Lambda$ is homeomorphic to a toral solen\-oid, and $\restrict f \Lambda$ is conjugate to a solen\-oidal automorphism. 

Furthermore we have $\stab x \cap \Lambda$ is perfect for every $x\in \Lambda$.  
Thus if $\stab x \cap \Lambda$ is locally connected, $\Lambda$ cannot be expanding, which by the above implies $\Lambda$ is homeomorphic to a solen\-oid.  However, the only  locally connected toral solen\-oids are in fact tori, that is, $\sol_A$ for $\det A = \pm 1$. 
\end{proof}


\section{Proof of Corollary \ref{cor:dim3} and Theorem \ref{thm:HA3D}}
The following observation is straightforward (see, for example, \cite[Lemma 2.4]{Aaron}).
\begin{lemma}\label{lem:trivialsplitting}
Let $\Lambda\subset M$ be a  compact hyperbolic set for a diffeomorphism $g\colon M\to M$. The  points $y \in \Lambda$ with the property that $E^\sigma(y) = \{0\}$ for some $\sigma \in \{s,u\}$ are periodic and isolated in $\Lambda$.  In particular if $\Lambda$ is transitive and contains such a point then $\Lambda$ is finite.  \end{lemma}

We now prove the remaining results from the introduction.
\begin{proof}[Proof of Corollary \ref{cor:dim3}]
By \cite[Theorem 3]{MR0286134} and by passing to the inverse if necessary we may assume that $\Lambda$ is an attractor for $f$. We thus have $\dim \restrict {E^u} {\Lambda} \le 2$.    Furthermore, by Lemma  \ref{lem:trivialsplitting}, $\dim \restrict {E^u} {\Lambda} = 0$ would  imply that $\dim(\Lambda) = 0$, whence we have $\dim \restrict {E^u} {\Lambda} \ge 1$.  
By the spectral decomposition and by passing to an appropriate iterate $f^n$ we may assume that $\Lambda$ is topologically mixing for $f^n$.

If $\dim \restrict {E^u} {\Lambda} =2 $  then $\Lambda$ is a co\-dimen\-sion-1 expanding attractor.  
If $\dim \restrict {E^u} {\Lambda} =1$ then by Theorem \ref{thm:main} we have that $\Lambda$ is an embedded toral solen\-oid.  By  \cite[Theorem 1]{MR2415057}, no proper 2-dimensional solen\-oid may be embedded in a closed orientable 3-man\-ifold.  

If needed, we first  argue on a double cover in the case $M$ is non-orientable.  Also if needed we may pass to a compact manifold with boundary $N$ containing $\Lambda$, and glue a second copy of $N$ along the boundary to obtain a closed manifold containing $\Lambda$.  
We may then apply \cite[Theorem 1]{MR2415057} and thus obtain that $\Lambda$ is locally connected, hence $\Lambda \cong \T^2$ and $f^n$ is conjugate to a toral automorphism.  
\end{proof}

\begin{proof}[Proof of Theorem \ref{thm:HA3D}]
We note that for a hyperbolic attractor we always have ${\dim \restrict {E^u} \Lambda \le \dim \Lambda}$. 

If $\dim \Lambda =0$ we have that $\dim \restrict {E^u} \Lambda= 0 $ hence Lemma   \ref{lem:trivialsplitting}  implies that every $x\in \Lambda$ is periodic and isolated;  hence we must have $\Lambda = \{x\}$ in order for $\Lambda$ to be topologically mixing.  
If $\dim \Lambda = 1$ then Lemma  \ref{lem:trivialsplitting}  implies  $\dim \restrict {E^u}\Lambda = 1 $; hence $\Lambda$ is an expanding \mbox{1-dimensional} attractor and  $\Lambda$ is conjugate to the shift map on a generalized \mbox{$1$-solen\-oid} by Theorem \ref{thm:Williams}.

If $\dim \Lambda = 2$, Lemma \ref{lem:trivialsplitting} implies  $1\le \dim \restrict {E^u}\Lambda \le 2 $.
When $\dim \restrict {E^u}\Lambda = 1$ the fact that $\Lambda$ is topologically mixing implies $\Lambda$ is connected, whence $\Lambda$ is homeomorphic to $\T^2$ and $\restrict f \Lambda$ is conjugate to a hyperbolic toral automorphism by Corollary \ref{cor:dim3}.  When $\dim \restrict {E^u}\Lambda = 2 $ then $\Lambda$ is a co\-dimen\-sion-1 expanding attractor by definition.  

Finally, when $\dim \Lambda = 3$ then Lemma  \ref{lem:trivialsplitting}  implies  $1\le \dim \restrict {E^u}\Lambda \le 2 $.  Furthermore, \cite[Theorem 4.3]{MR0006493} implies $\Lambda$ has non-empty interior, which by \cite[Theorem 1]{MR2199438} implies $\Lambda = M$.  But then the result follows from Theorem \ref{thm:fn}.
\end{proof}


\begin{thebibliography}{99}

\bibitem{BonattiEmail}
C.~Bonatti,
\newblock {Problem in dynamical systems},
\newblock \url{http://www.math.sunysb.edu/dynamics/bonatti_prob.txt}, 
\newblock November 1999.

\bibitem{MR1179171} 
\newblock H.~G. Bothe,
\newblock {\em {Expanding attractors with stable foliations of class {$C^0$}}},
\newblock in ``Ergodic theory and related topics, {III}," Lecture Notes in Math.,  {\bf 1514}, Springer, Berlin, 1992, pp. 36--61.

\bibitem{MR1407604} 
\newblock B.~Brenken,
\newblock {\em {The local product structure of expansive automorphisms of
  solen\-oids and their associated {$C^\ast$}-algebras}},
\newblock Canad. J. Math., {\bf 48} (1996), 692--709.

\bibitem{Aaron}
\newblock A.~Brown,
\newblock {\em {Constraints On Dynamics Preserving Certain Hyperbolic Sets}},
\newblock Ergodic Theory Dynam. Systems,  to appear.

\bibitem{MR2199438} 
\newblock T.~Fisher,
\newblock {\em {Hyperbolic sets with nonempty interior}},
\newblock Discrete Contin. Dyn. Syst., {\bf 15} (2006), 433--446.

\bibitem{MR0271990} 
\newblock J.~Franks,
\newblock {\em {Anosov diffeomorphisms}},
\newblock in ``Global {A}nalysis," Amer. Math. Soc., Providence, R.I.,
  1970, pp. 61--93.

\bibitem{MR2249032} 
\newblock V.~Z. Grines, V.~S. Medvedev, and E.~V. Zhuzhoma,
\newblock {\em {On surface attractors and repellers in 3-man\-ifolds}},
\newblock Mat. Zametki, {\bf 78} (2005), 813--826.

\bibitem{MR1254137} 
\newblock B.~G{\"u}nther,
\newblock {\em {Attractors which are homeomorphic to compact abelian groups}},
\newblock Manuscripta Math., {\bf 82} (1994), 31--40.

\bibitem{MR1836432} 
\newblock K.~Hiraide,
\newblock {\em {A simple proof of the {F}ranks-{N}ewhouse theorem on
  co\-dimen\-sion-one {A}nosov diffeomorphisms}},
\newblock Ergodic Theory Dynam. Systems, {\bf 21} (2001), 801--806.

\bibitem{MR0006493} 
\newblock W.~Hurewicz and H.~Wallman,
\newblock ``Dimension {T}heory,''
\newblock Princeton University Press, Princeton, N. J., 1941.

\bibitem{MR2415057}
\newblock B.~Jiang, S.~Wang, and H.~Zheng,
\newblock {\em {No embeddings of solenoids into surfaces}},
\newblock Proc. Amer. Math. Soc., {\bf 136} (2008), 3697--3700.

\bibitem{MR766105} 
\newblock J.~L. Kaplan, J.~Mallet-Paret, and J.~A. Yorke,
\newblock {\em {The {L}yapunov dimension of a nowhere differentiable attracting
  torus}},
\newblock Ergodic Theory Dynam. Systems, {\bf 4} (1984), 261--281.

\bibitem{MR1326374}
\newblock A.~Katok and B.~Hasselblatt,
\newblock ``Introduction to the modern theory of dynamical systems,''
\newblock Cambridge University Press, Cambridge, 1995.


\bibitem{Manning:1973p11377} 
\newblock A.~Manning,
\newblock {\em {There are no new {A}nosov diffeomorphisms on tori}},
\newblock Amer. J. Math., {\bf 96} (1974), 422--429.

\bibitem{MR0277004} 
\newblock S.~E. Newhouse,
\newblock {\em {On codimension one {A}nosov diffeomorphisms}},
\newblock Amer. J. Math., {\bf 92} (1970), 761--770.



\bibitem{MR0286134}
\newblock R.~V. Plykin,
\newblock {\em {The topology of basic sets of {S}male diffeomorphisms}},
\newblock Math. USSR-Sb., {\bf 13}(2) (1971), 297--307.

\bibitem{MR580625}
\newblock R.~V. Plykin,
\newblock {\em {Hyperbolic attractors of diffeomorphisms}},
\newblock Russian Math. Surveys, {\bf 35}(3) (1980), 109--121.

\bibitem{MR586207}
\newblock R.~V. Plykin,
\newblock {\em {Hyperbolic attractors of diffeomorphisms (the nonorientable
  case)}},
\newblock Russian Math. Surveys, {\bf 35}(4) (1980), 186--187.

\bibitem{MR0415679} 
\newblock D.~Ruelle and D.~Sullivan,
\newblock {\em {Currents, flows and diffeomorphisms}},
\newblock Topology, {\bf 14} (1975), 319--327.

\bibitem{MR0228014}
\newblock S.~Smale,
\newblock {\em {Differentiable dynamical systems}},
\newblock Bull. Amer. Math. Soc., {\bf 73} (1967), 747--817.

\bibitem{MR0217808}
\newblock R.~F. Williams,
\newblock {\em {One-dimensional non-wandering sets}},
\newblock Topology, {\bf 6} (1967), 473--487.

\bibitem{MR0266227}
\newblock R.~F. Williams,
\newblock {\em {Classification of one dimensional attractors}},
\newblock in ``Global {A}nalysis," Amer. Math. Soc., Providence, R.I., 1970, pp. 341--361.

\bibitem{MR0348794}
\newblock R.~F. Williams,
\newblock {\em {Expanding attractors}},
\newblock Inst. Hautes \'Etudes Sci. Publ. Math.,  (1974), 169--203.

\end{thebibliography}
\end{document}